\documentclass[a4paper,11pt,reqno]{amsart}
\usepackage{amsmath,amsfonts,enumerate,amssymb,amsthm,color}
\usepackage{pgf,tikz}

\title{Elementary abelian groups of rank $5$ are DCI-groups}

\author{Yan-Quan~Feng}
\address{Y.-Q.~Feng,
Department of Mathematics, Beijing Jiaotong University, Beijing 100044 PR China}
\email{yqfeng@bjtu.edu.cn}

\author{Istv\'an~Kov\'acs}
\address{I.~Kov\'acs,
IAM and FAMNIT, University of Primorska, Glagolja\v{s}ka 8, 6000 Koper, Slovenia}
\email{istvan.kovacs@upr.si}

\thanks{
\\ [+0.75ex] Y.-Q.~Feng was supported in part by the National Natural Science Foundation of China (11571035, 11231008) and the 111 Project of China (B16002). I.~Kov\'acs was supported in part by
the Slovenian Research Agency (research program P1-0285, and
research projects N1-0032, N1-0038, J1-5433 and J1-6720) and is grateful
to Beijing Jiaotong University for hospitality.}

\newtheorem{thm}{Theorem}[section]
\newtheorem{lem}[thm]{Lemma}
\newtheorem{prop}[thm]{Proposition}
\newtheorem{cor}[thm]{Corollary}
\theoremstyle{definition}
\newtheorem{defi}[thm]{Definition}
\newtheorem{hyp}[thm]{Hypothesis}
\theoremstyle{remark}
\newtheorem{rem}[thm]{Remark}

\def\A{\mathcal{A}}
\def\B{\mathcal{B}}

\def\id{\mathrm{id}}
\def\O{\mathbf{O}_\theta}
\def\Q{\mathbb{Q}}
\def\un{\underline}
\def\teq{\approx_2}
\def\Z{\mathbb{Z}}

\DeclareMathOperator{\Aut}{Aut}
\DeclareMathOperator{\Bs}{Bsets}
\DeclareMathOperator{\Cay}{Cay}

\DeclareMathOperator{\Iso}{Iso}
\DeclareMathOperator{\Orb}{Orb}
\DeclareMathOperator{\rad}{rad}
\DeclareMathOperator{\rank}{rank}

\DeclareMathOperator{\Sym}{Sym}
\DeclareMathOperator{\Fun}{Fun}
\newcommand{\sg}[1]{\langle #1 \rangle}
\newcommand{\sgg}[1]{\langle\!\langle {#1}  \rangle\!\rangle}

\begin{document}

\begin{abstract}
In this paper, we show that the group $\Z_p^5$ is a DCI-group for any odd prime $p,$
that is, two Cayley digraphs $\Cay(\Z_p^5,S)$ and $\Cay(\Z_p^5,T)$
are isomorphic if and only if $S=T^\varphi$ for some automorphism
$\varphi$ of the group $\Z_p^5$.
\end{abstract}

\maketitle

\section{Introduction}

Let $H$ be a finite group and $S$ be a subset of $H$. The
{\em Cayley digraph} $\Cay(H,S)$ is the digraph that has vertex set $H,$ and arc set $\{ (x,sx) : x \in H, s\in S\}$.  It follows from the definition that $\Aut(\Cay(H,S))$ contains $H_R,$ the group of all {\em right translations} $H_R=\{h_R : h \in H\},$ where
$x^{h_R}=xh, \, x \in H$. Also, $\Cay(H,S)$ is loopless if the identity element $1 \notin S,$ and it is regarded as an undirected graph when $S$ is an inverse-closed set, that is, $S=S^{-1}=\{x^{-1} : x \in S\}$.

Two Cayley digraphs $\Cay(H,S)$ and $\Cay(H,T)$ are called {\em Cayley isomorphic} if $T=S^\varphi$ for some automorphism $\varphi \in \Aut(H)$.  It is trivial to show that Cayley isomorphic Cayley digraphs are isomorphic as digraphs. The converse, however, does not hold in general. There are examples of Cayley digraphs which are isomorphic but not Cayley isomorphic. A subset $S \subseteq H$ is called a {\em CI-subset} if for any $T \subseteq H,$ the isomorphism
$\Cay(H,T) \cong \Cay(H,S)$ implies that $T=S^\varphi$ for some $\varphi \in \Aut(H)$. The group $H$ is a {\em DCI-group} if each of its subsets are CI-subsets, and a {\em CI-group} if each of its  inverse-closed subsets are CI-subsets. Motivated by a problem posed by \'Ad\'am in \cite{A}, Babai and Frankl~\cite{BF} asked the following question: Which are the CI-groups?
Although the candidates of CI-groups have been reduced to a restricted list \cite{DMS,LLP}, which was obtained by accumulating the work of several mathematicians, it is considered to be difficult to confirm that a particular group is a CI-group. We refer the reader to the survey paper \cite{L2} for most results on CI- and DCI-groups.

One of the crucial steps towards the classification of all CI-groups is to answer which elementary
abelian $p$-groups are CI-groups (see also \cite[Question~8.3]{L2}). It is known that the group
$\Z_p^n$ is a CI-group in each of the following cases: $n=1$ \cite{Dj,ET,T};
$n=2$ \cite{AN,G}; $n=3$ \cite{AN,D}; $n=4$ and $p=2$ \cite{CL}; $n=4$ and $p>2$ \cite{HM}
(a proof for $n=4$ with no condition on $p$ was given recently in \cite{Mo}); $n=5$ and $p=2$ \cite{CL}; and $n=5$ and $p=3$ \cite{Sp9}. On the other hand, some examples of groups $\Z_p^n$ are also known which are not CI-groups,
and in each case the rank $n \ge 6$. Nowitz~\cite{N} found a non CI-subset of $\Z_2^6,$ and more recently, Spiga~\cite{Sp9} constructed a non-CI subset of $\Z_3^8$. Constructions of non-CI subsets of $\Z_p^n$ where $n$ is expressed as a function in $p$ were the subject of the papers \cite{Mu,So,Sp7}. The best bound is due to Somlai~\cite{So}, which says that $\Z_p^n$ is not a CI-group if $n \ge  2p+3$. The question whether $\Z_p^5$ is a CI-group for any odd prime $p$ is mentioned in \cite{L2} as a crucial task for classifying CI-groups (see Section~8.4 and Problem~8.10). The goal of this paper is to complete this task by proving the following theorem:

\begin{thm}\label{1}
The group $\Z_p^5$ is a DCI-group for any odd prime $p$.
\end{thm}

Our starting point is the following group theoretical criterion due to Babai~\cite{B}: A subset $S \subseteq H$ is a CI-subset if and only if any two regular subgroups of $\Aut(\Cay(H,S))$ isomorphic to
$H$ are conjugate in $\Aut(\Cay(H,S))$. Recall that, the group $H_R$ of right translations is always contained in $\Aut(\Cay(H,S))$. Motivated by this criterion, the following
definition was introduced by Hirasaka and Muzychuk~\cite{HM}: A permutation group $G \le \Sym(\Omega)$ containing a fixed subgroup $F$ is {\em $F$-transjugate} if for each $g\in \Sym(\Omega),$ the condition
$g^{-1}F g \le G$ implies that $g^{-1}F g$ and $F$ and are conjugate in $G$. In this context, 
Babai's result can be rephrased as to say that a subset $S \subseteq H$ is a CI-subset if and
only if the group $\Aut(\Cay(H,S))$ is $H_R$-transjugate.
It is well-known that $\Aut(\Cay(H,S))$ is a $2$-closed permutation group for any $S \subseteq H$
(for the definition of a $2$-closed permutation group, see Section~2.1). Following \cite{HM}, we say that $H$ is a {\em CI$^{(2)}$-group} if all $2$-closed subgroups of $\Sym(H)$ containing $H_R$ are $H_R$-transjugate. Clearly, if $H$ is a CI$^{(2)}$-group, then it is necessarily a DCI-group. In fact, instead of Theorem~\ref{1} we prove the following slightly more general theorem:

\begin{thm}\label{2}
The group $\Z_p^5$ is a CI$^{(2)}$-group for any odd prime $p$.
\end{thm}

We prove Theorem~\ref{2} following the so called S-ring approach (S-ring is the abbreviation of {\em Schur ring}, and
for a definition, see Section~2.2). Roughly speaking, S-rings are certain subalgebras of the group algebra $\Q H$ which were introduced by  Schur~\cite{Sch} in order to study permutation groups containing a regular subgroup isomorphic to $H$. The usage of S-rings in the investigation of CI-groups was proposed by Klin and P\"oschel~\cite{KP80,KP81}.
For a concise survey on S-rings and their applications in combinatorics, we refer the reader to \cite{MPon}.
\medskip

We finish the introduction with a brief outline of the paper:
Section~2 contains preliminary material, especially, a thorough introduction to S-ring theory.
We intend to keep our text as self-contained as possible.
In Sections~3, we turn to S-rings over elementary abelian $p$-groups of arbitrary rank. In particular, an equivalent condition will be derived for the group $\Z_p^n$ to be a CI$^{(2)}$-group in terms of its S-rings (Proposition~\ref{P-iff}). We remark that, this condition is obtained by combining together several results
proved in \cite{HM,Mo,Sp7}. Based on this equivalence, Theorem~\ref{2} will be reformulated in a statement
involving a particular class of S-rings over the group $\Z_p^5$ (Theorem~\ref{3}).
Then, in Section~4, we derive a property of S-rings over $\Z_p^4$ which will be needed when dealing with S-rings over $\Z_p^5$.
The proof of Theorem~\ref{3} will be divided into two parts depending on whether the S-rings in question are
decomposable or not (for a definition of a decomposable S-ring, see Section~2.2). The decomposable S-rings will be handled in Section~5, while the indecomposable ones in Section~6.

\section{Preliminaries}

All groups in this paper are finite.
In this section we collect all concepts and facts needed in this paper.

\subsection{Permutation groups}

Let $G \le \Sym(\Omega)$ be a permutation group of a finite  set
$\Omega$. For $\omega \in \Omega,$ we denote by $G_\omega$ the {\em stabilizer} of $\omega$ in $G,$ and by $\omega^G$ the
{\em $G$-orbit} of $\omega$.
For a subset $\Delta \subseteq \Omega$ and permutation $\gamma \in \Sym(\Omega),$
we say that $\gamma$ fixes $\Delta$ if $\Delta^\gamma=\Delta,$ and
that $\gamma$ fixes $\Delta$ pointwise if $\omega^\gamma=\omega$ for all $\omega \in \Delta$. The {\em setwise stabilizer}
and {\em pointwise stabilizer} of $\Delta$ in $G$ will be denoted by $G_{\{\Delta\}}$ and $G_\Delta,$ resp.,
that is, $G_{\{\Delta\}}=\{g\in G : \Delta^g=\Delta \}$ and $G_\Delta=\{g\in G : \omega^g=\omega, \,
\omega\in \Delta \}$.
The set of all $G$-orbits is denoted by $\Orb(G,\Omega)$. Suppose that $G$ is transitive on $\Omega$. If $\delta=\{\Delta_1,\ldots,\Delta_n\}$ is a {\em block system} for $G,$ then we write $G_\delta$ for the kernel of the action of $G$ on $\delta,$ and $G^\delta$ for the permutation group of $\delta$ induced by $G$.

Two permutation groups  $H,G \le \Sym(\Omega)$ are said to be {\em $2$-equivalent}, denoted by $H \teq G,$ if $\Orb(H,\Omega^2)=\Orb(G,\Omega^2),$ see \cite{W69}.
The equivalence class of $G$ contains a largest subgroup, which is called the {\em $2$-closure} of $G$, denoted by $G^{(2)}$.
The group $G$ is called {\em $2$-closed} if $G^{(2)}=G$.

\begin{prop}\label{P-center-G2}
For any $G \le \Sym(\Omega),$ $Z(G) \le Z(G^{(2)})$.
\end{prop}

\begin{proof}
Let  $g_1 \in Z(G), \, \gamma \in G^{(2)}$ and
$\omega \in \Omega$.  Since  $\Orb(G,\Omega^2)=\Orb(G^{(2)},\Omega^2)$, we have $(w,w^{g_1})^G= (w,w^{g_1})^{G^{(2)}},$ and so there exists $g \in G$, depending on $w$ and $w^{g_1},$ such that $(w,w^{g_1})^g= (w,w^{g_1})^\gamma$.
Then $\omega^{g_1\gamma}=\omega^{g_1g}=\omega^{gg_1}=
\omega^{\gamma g_1}$.  As $\omega$ was chosen arbitrarily
from $\Omega$ and $\gamma$ from $G^{(2)},$
it follows that $g_1 \in Z(G^{(2)}),$ hence
$Z(G) \le Z(G^{(2)}),$ and the assertion follows.
\end{proof}

A transitive permutation group $G$ and its $2$-closure $G^{(2)}$ have the same block systems (see \cite[Theorem~4.11]{W69}).

\begin{prop}{{\rm (\cite[Proposition~2.1]{HM})}}\label{HM1}
Let $G \le \Sym(\Omega)$ be a transitive permutation group,
and let $\delta$ be a block system for $G$. Then
\begin{enumerate}[(i)]
\item $(G^{(2)})^\delta \le (G^\delta)^{(2)}$.
\item If $G$ is $2$-closed and $F^\delta$ is $2$-closed for some
$F \le G,$ then $FG_\delta$ is also $2$-closed.
\end{enumerate}
\end{prop}

The following statement is given as Exercise~5.8 in \cite{W69}.
It also appears as 5.1.~Proposition in the preprint \cite{MPos}, where the authors give a
proof. Regarding the fact that \cite{MPos} is a university preprint, we also present a proof.

\begin{prop}{{\rm (cf.~\cite{W69})}}\label{W}
If $G \le \Sym(\Omega)$ is a $p$-group, then
$G^{(2)}$ is also a $p$-group.
\end{prop}

\begin{proof}
If $G$ is intransitive, then $G^{(2)}$ is a subdirect product of the $2$-closures
of the transitive components of $G$. Thus, it is sufficient to prove the theorem
for transitive groups $G \le \Sym(\Omega)$.
Suppose to the contrary that $G$ is a counterexample to the proposition
whose order is the smallest possible.  If $G$ is abelian, then it is regular on $\Omega$.
By Proposition~\ref{P-center-G2}, $G^{(2)}$  centralizes $G,$ and it follows that
$G^{(2)}=G$. Thus, we may assume that $G$ is non-abelian, and so 
$|G| \ge p^3$. 
The center $Z(G)$ is nontrivial. Let $P \le Z(G)$ such that $|P|=p,$ and let
$g \in G^{(2)}$ be an element of order $q$  for some prime $q \ne p$.  Then $\Orb(P,\Omega)$ is a block system of $G$ on $\Omega$.
Let us consider the natural action of $G$ on $\Orb(P,\Omega)$.
To simplify notation, we write $\delta=\Orb(P,\Omega)$. The permutation group $G^\delta \le \Sym(\delta)$ induced by
$G$ is transitive on $\delta$ and has order less than $|G|$. By the  minimality  of $G$ the group $(G^{\delta})^{(2)}$ is a $p$-group. Thus $(G^{(2)})^\delta$ is also
a $p$-group, see Proposition~\ref{HM1}(i),  which implies that $g$
acts on $\delta$ as the identity permutation.
Equivalently, $g$ fixes any $P$-orbit $\Delta \in \delta$.
By Proposition~\ref{P-center-G2},  $P \le Z(G^{(2)}),$ hence $g$ centralizes $P$. This implies that $g$ is
semiregular on $\Delta$. Since $g$ has order $q$ and $|\Delta|=p,$ $g$ fixes
pointwise $\Delta,$ and as this is true for any $\Delta \in \delta,$ $g$ is the identity permutation of $\Omega,$ a contradiction.
\end{proof}

\begin{prop}{{\rm( \cite[Proposition~3.6(ii)]{HM} )}}\label{HM2}
Let $H$ be an abelian $p$-group whose order $|H| \ge p^3,$ and
let $G \le \Sym(H)$ with $G \ge H_R$. If there exists a $G_1$-orbit $T$ such
that $|T|=p$ and $\sg{T}=H,$ then $|G|=p \cdot |H|$.
\end{prop}

Finally, we recall a recent result of Morris~\cite{Mo}.
Let $H \cong \Z_p^n$ for an arbitrary prime $p$.
Assume that $G \le \Sym(H)$ is a $p$-group such that
$$
G=\big\langle H_R,\pi^{-1} H_R\pi \big\rangle \text{ for some } \pi \in \Sym(H).
$$
Let $P$ be a Sylow $p$-subgroup of $\Sym(H)$ with $G \le P$.
Then $P$ is permutation isomorphic to  the iterated
wreath product $\Z_p \wr \cdots \wr \Z_p$
($n$ copies of $\Z_p$), and this shows that $P$
admits block systems $\delta_0,\ldots,\delta_{n-1}$ such that
$\delta_i$ has blocks of size $p^{i+1},$ and if $0 \le i < j \le n-1,$ then
each class of $\delta_i$ is contained in a class of $\delta_j$.
Since $H_R$ is abelian, the kernel $(H_R)_{\delta_i}$ has order
$p^{i+1}$.  In particular, there exist  $\tau_0, \, \tau_1 \in H_R$
such that $(H_R)_{\delta_0}=\sg{\tau_0}$ and $(H_R)_{\delta_1}=\sg{\tau_0,\tau_1}$. Note that, we can write $\delta_0=\Orb(\sg{\tau_0},H)$ and $\delta_1=\Orb(\sg{\tau_0,\tau_1},H)$.

\begin{prop}{{\rm (\cite[Corollary~3.2]{Mo})}}\label{M}
With the above notation, there exists $\psi \in G^{(2)}$
such that $\psi$ commutes with $\tau_0,$ and
$\psi^{-1}\pi^{-1} H_R \pi \psi$ contains $\tau_1$.
\end{prop}

Let us consider once more the above groups $G=\sg{H_R,\, \pi^{-1} H_R \pi}$ and $P,$ where $P$
is a Sylow $p$-subgroup of $\Sym(H)$ with $G \le P$.
Then $Z(P) \leq \pi^{-1} H_R \pi$ because $\pi^{-1} H_R \pi$ is
abelian and regular on $H$. Similarly, $Z(P)\leq H_R $. On the other hand, $P_{\delta_0} \lhd P$ and $\tau_0 \in P_{\delta_0}$, implying $P_{\delta_0}\cap Z(P)\not =1$ because $P$ is $p$-group. Then $P_{\delta_0}\cap Z(P)\leq  H_R$ implies $P_{\delta_0}\cap Z(P)=
\sg{\tau_0},$ hence $\tau_0 \in P_{\delta_0}\cap Z(P)\leq
\pi^{-1} H_R \pi$. Proposition~\ref{M} together with the condition
that $\psi$ centralizes $\tau_0$ shows that
$\tau_0,\tau_1 \in \psi^{-1}\pi^{-1} H_R \pi \psi,$ and hence
$|C_{H_R}(\psi^{-1}\pi^{-1} H_R \pi \psi)|\geq p^2$.
This inequality will be used later.

\subsection{S-rings}

In this subsection, we give the definition of an S-ring, and
review several basic properties.
Let $H$ be a finite group with identity element $1,$ and let $\Q H$ denote the group algebra of $H$ over the rational number field. For a subset $T \subseteq H,$ we define
the $\Q H$-element $\un{T}$ as the formal sum $\un{T}=\sum_{h\in H}a_h h$ with $a_h=1$ if $h\in T,$ and $a_h=0$ otherwise. We remark
that the $\Q H$-element $\un{T}$ is traditionally called {\em simple quantity}, see \cite{W64}.

By a {\em Schur-ring} over $H$ ({\em S-ring} for short) we mean a subalgebra $\A \subseteq \Q H,$ which can be associated with a partition $\pi$ of $H$ satisfying the following conditions:
\begin{itemize}
\item  The set $\{1\}$ belongs to $\pi$.
\item For every $T\in \pi$, the set $T^{-1}$ belongs to  $\pi$.
\item $\A$ is spanned by the $\Q H$-elements
$\underline{T}, \, T \in \pi$.
\end{itemize}

The elements (classes) of $\pi$ are also called the {\em basic sets} of $\A$, and
from now on we will use the notation $\Bs(\A)$ for $\pi$.
The cardinality $|\Bs(\A)|$ is called the {\em rank} of $\A$.
The concept of S-ring is due to Wielandt~\cite{W64}, which was motivated by the following result of Schur~\cite{Sch}:

\begin{thm}{{\rm (cf.~\cite[Theorem~24.1]{W64})}}
Let $H$ be a finite group, and let $G \le \Sym(H)$ with $H_R \le G$.
Then the $\Q H$-elements $\un{T}, \, T \in  \Orb(G_1,H)$ span an S-ring
over $H$.
\end{thm}

The S-ring in the above theorem is also called the
{\em transitvity module} of $G_1,$ denoted by
$V(H,G_1)$. We note that, there exist S-rings which do not arise as transitivity
modules. Given an arbitrary S-ring $\A$ over a group $H,$ we say that
$\A$ is {\em Schurian} if $\A=V(H,K_1)$ for some permutation group $K \le \Sym(H)$ with $H_R \le K,$ and that $\A$ is {\em non-Schurian} otherwise.

\begin{rem} It should be noted that the pair $( H, \{ \Cay(H,T) : T\in \Bs(\A)\} )$
forms a  {\em Cayley (association) scheme} in the sense of \cite{MPon}.
Thus, S-ring theory can be regarded as a part of the theory of association schemes,
and several concepts defined for S-rings can be understood in this context.
\end{rem}

Let $\A$ be any S-ring over a group $H$.
A subset $S \subseteq H$ (subgroup $K \le H$, resp.) is called an {\em $\A$-subset} ({\em $\A$-subgroup}, resp.) if $\un{S} \in \A$ ($\un{K} \in \A$, resp.). The \emph{radical} of a subset $S \subseteq H$ is the subgroup of $H$ defined as
$$
\rad(S)=\{h \in H : hS=Sh=S\}.
$$
In other words, $\rad(S)$ is the largest subgroup $E \le H$ for
which $S$ is equal to the union of some left $E$-cosets and also
some right $E$-cosets.
If $S$ is an $\A$-subset, then both groups $\rad(S)$ and $\sg{S}$  are $\A$-subgroups
(see \cite[Propositions~23.5 and 23.6]{W64}).
If $K,L \le H$ are two  $\A$-subgroups, then it can be easily checked that both
$K \cap L$ and $\sg{K \cup L}$ are $\A$-subgroups.
The {\em thin radical} of $\A$ is defined as
$$
\O(\A)=\{ g \in G : \{g\} \in \Bs(\A) \}.
$$
The following simple, but useful property is a simple observation:
\begin{equation}\label{Eq-eT}
\text{ If } e\in \O(\A) \text{ and } T \in \Bs(\A), \text{ then both sets }eT \text{ and } Te \text{ are in }\Bs(\A).
\end{equation}
Here $eT=\{et: t \in T\}$ and $Te=\{te : t \in T\}$. It follows that  
the thin radical $\O(\A)$ is an $\A$-subgroup.

Let $K \le H$ be an $\A$-subgroup. Then, for any basic set $T \in \Bs(\A)$ there exist positive integers $k$ and $k'$ such that
\begin{equation}\label{Eq-T/K}
|Kh \cap T|=k \text{ and } |hK \cap T|=k' \text{ \ for all } h \in T.
\end{equation}
These can be verified by considering the products
$\un{K}\cdot \un{T}$ and  $\un{T}\cdot \un{K}$.
Here and in what follows the symbol $\cdot$ denotes the
multiplication of  $\Q H$. Since both products belong to $\A,$ these
can be expressed as a linear combination of the simple quantities
$\un{T'}, \, T' \in \Bs(\A)$. The coefficient by
$\un{T}$ is equal to $k$ in the case of $\un{K}\cdot\un{T},$ and
it is equal to $k'$ in the case of $\un{T}\cdot\un{K}$.

The subalgebra $\Q  K \cap \A$ is an S-ring over an $\A$-subgroup $K,$ denoted by
$\A_K,$ which is called the {\em S-subring of $\A$ induced by $K$}.

Assume, in addition, that $K \trianglelefteq H$ for an $\A$-subgroup $K$.
For a subset $S \subseteq H,$ we let $S/K=\{Kh,\, h \in S\}$.
Let $T_1,T_2 \in \Bs(\A)$ such that $KT_1 \cap KT_2 \ne \emptyset,$ or equivalently,
$k_1t_1=k_2t_2$ holds for some $k_i \in K$ and $t_i\in T_i$ ($i=1,2$).
This shows that, the coefficient $a_{T_1} >0$ by the linear combination
$\un{K}\cdot \un{T_2}=\sum_{T \in \Bs(\A)}a_T \un{T}$. This implies
that $T_1 \subseteq KT_2,$ and hence $KT_1 \subseteq KT_2$.
Similarly, $b_{T_2} >0$ by $\un{K}\cdot \un{T_1}=\sum_{T \in \Bs(\A)}b_T \un{T},$
hence $T_2 \subseteq KT_1,$ and so $KT_2 \subseteq KT_1$ also holds.
We conclude that $KT_1=KT_2,$ and thus the sets $KT, T \in \Bs(\A)$ form a
partition of $H,$ and consequently, the sets $T/K, \, T\in \Bs(\A)$ form a partition of
$H/K$. The corresponding $\Q \, H/K$-elements $\un{T/K}$ span an S-ring over $H/K,$ which is called the {\em quotient of $\A$ by $K,$} denoted by $\A_{H/K}$.

Assume that  $H=E \times F$ is the internal direct product of its subgroups $E$ and $F,$
and that $\A$ is an S-ring over $H$ such that both
$E$ and $F$ are $\A$-subgroups. Since $E \cap F=\{1\},$
it follows that $\un{XY}=\un{X} \cdot \un{Y}$ for any subsets $X \subseteq E$ and $Y \subseteq F$. A straightforward computation yields that the simple quantities
$\un{RS},\, R \in \Bs(\A_E), \, S \in \Bs(\A_F)$ span an S-ring over $H$.
The latter S-ring is called the {\em tensor product} of $\A_E$ with $\A_F,$
denoted by $\A_E \otimes \A_F$.  Clearly, $\A_E \otimes \A_F=\A_F \otimes \A_E,$
and $\A_E \otimes \A_F \subseteq \A$. The following lemma can be
easily shown using Eq.~\eqref{Eq-eT}.

\begin{lem}\label{L-tensor}
Let $\A$ be an S-ring of the internal direct product $H=E \times F$ such that
both $E$ and $F$ are $\A$-subgroups. If $\A_E=\Q E$ or $\A_F=\Q F,$
then $\A=\A_E \otimes \A_F$.
\end{lem}

The following result is also known as Schur's first theorem on multipliers (see \cite{MPon}).

\begin{thm}{\rm (cf.\ \cite[Theorem~23.9(a)]{W64})}\label{T-1st-multi}
Let $\A$ be an S-ring over an abelian group $H,$ $T \in \Bs(\A)$ be any basic set, and suppose that $k$ is an integer coprime to $|H|$.
Then the set  $T^{(k)}:=\{ h^k : h \in T\}$ is a basic set  of $\A$.
\end{thm}

Finally, we recall the concept of {\em $E/F$-wreath product} after \cite{MPon}.
This was defined in \cite{LM} under the name {\em wedge product}, and independently in \cite{EP} under the name {\em generalized wreath product}.
Let $\A$ be an S-ring over a group $H$. If there exist $\A$-subgroups $E$ and $F$ such that
$$
F \le E, \; F \triangleleft H, \text{ and }
F \le  \rad(T) \text{ for all } T\in \Bs(\A), \, T \subset H \setminus E,
$$
then we say that $\A$ is an {\em $E/F$-wreath product} and
write $\A=\A_E \wr_{E/F} \A_{H/F}$. Note that, the S-ring $\A$ can be reconstructed uniquely from the S-rings  $\A_E$ and $\A_{H/F}$.  In the particular case when $E=F,$
we use term {\em wreath product}, and write
$\A_E \wr \A_{H/E}$ for  $\A_E \wr_{E/E} \A_{H/E}$.
 In what follows we say that $\A$ is {\em decomposable} if it can be decomposed as $\A=\A_E \wr_{E/F} \A_{H/F}$ where $E \ne H$ and $F \ne \{1\},$ and that $\A$ is
{\em indecomposable} otherwise.

\subsection{Automorphisms of S-rings}

Let $\A \subseteq \Q H$ be an S-ring over a group $H$.
By an {\em automorphism} of $\A$ we mean a permutation of $H$ that is an automorphism of all Cayley graphs $\Cay(H,T), \, T \in \Bs(\A)$.
This definition is due to Klin and P\"oschel~\cite{KP80} (see also \cite{MPon}). The group of all automorphisms of $\A$ will be denoted by $\Aut(\A),$ that is,
$$
\Aut(\A)=\bigcap_{T \in \Bs(\A)}\Aut(\Cay(H,T)).
$$

In what follows, we write $\Aut(\A)_1$ for the stabilizer $(\Aut(\A))_1$.
Note that, as a permutation group of $H,$ $\Aut(\A)$ is $2$-closed.
Moreover, if $\A=V(H,G_1)$ for some $G \le \Sym(H)$ with $G \ge H_R,$ then
$\Aut(V(H,G_1))=G^{(2)}$. If $K \le G$ is a subgroup with $H_R \le  K,$ then
$V(H,K_1) \supseteq V(H,G_1)$. Also, given two S-rings $\A$ and $\B$ of the same group $H,$ the inequality $\B \subseteq \A$ implies that
$\Aut(\B) \ge \Aut(\A)$.

For two arbitrary S-rings $\A,\B \subseteq \Q H,$ their intersection $\A \cap \B$
is also an S-ring over $H$ (cf. \cite{HM,MPon}).  Therefore, given any subset
$S \subset H,$ it is possible to define the S-ring
$\sgg{S}:=\cap_{\A^*}\A^*,$ where $\A^*$ runs over the set of all S-rings over
$H$ that contain $\un{S}$. Then, the following identity holds:
\begin{equation}\label{Eq-auto}
\Aut(\Cay(H,S))=\Aut(\sgg{S}).
\end{equation}
Indeed, let $G=\Aut(\Cay(H,S))$ and $\A=V(H,G_1)$.
The fact that $G \ge \Aut(\sgg{S})$ follows if we observe
that $S$ can be expressed as $S=\cup_{i=1}^{k}T_i$ for some basic sets
$T_i \in \Bs(\sgg{S}),$ and thus
$\Aut(\sgg{S}) \le \bigcap_{i=1}^{k}\Aut(\Cay(H,T_i)) \le G$.
On the other,  since $G$ is $2$-closed, $G=\Aut(\A)$. Also, as any element of
$G_1$ maps $S$ to itself,  it follows that $\un{S} \in \A$. This implies in turn that
$\sgg{S} \subseteq \A,$ and so $\Aut(\sgg{S}) \ge \Aut(\A)=G,$
and Eq.~\eqref{Eq-auto} follows.

Suppose that $K \le H$ is an $\A$-subgroup and write $G=\Aut(\A)$.
Then any element of the stabilizer $G_1$ maps $K$ to itself. This implies that
the setwise stabilizer $G_{\{K\}}$ factorizes as
$G_{\{K\}}=G_1 K_R$. In particular,  $G_1 \le G_{\{K\}},$ and hence the
$G_{\{K\}}$-orbit of $1$ is a block for $G$ (see \cite[Theorem~1.5A]{DM}). The latter orbit is $K,$ and we conclude that the induced block system $\delta=\{K^g : g \in G\}$
is equal to  the set $H/K$ of all right cosets of $K$ in $H$.

Finally, we point out a relation between $\Aut(\A)$ and the thin radical $\O(\A)$.  For $h \in H,$ the {\em left translation} $h_L \in \Sym(H)$ is the permutation acting as $x^{h_L}=h^{-1}x, \, x\in H$.
If $\A$ is a Schurian S-ring over $H,$ then its thin radical $\O(\A)$  satisfies the following:
\begin{equation}\label{Eq-thin}
\O(\A)=\{h \in H : h_L \in C_{\Sym(H)}(\Aut(\A)) \}.
\end{equation}
Indeed, if $h \in \O(\A)$ then every
$g \in \Aut(\A)$ acts as an automorphism of
$\Cay(H,\{h\})$. It is straightforward to check that this implies that $g$ and
$(h^{-1})_L$ commute, and so
$h_L \in C_{\Sym(H)}(\Aut(\A))$. On the other hand, if $h_L \in C_{\Sym(H)}(\Aut(\A))$ and
$g \in \Aut(\A)_1,$  then $(h^{-1})^g=1^{h_L g}=1^{g h_L}=h^{-1}$. Therefore, the  orbit of $h^{-1}$ under $\Aut(\A)_1$ is equal to the set $\{h^{-1}\}$. Now, since $\A$  is Schurian, its basic sets are the
$\Aut(\A)_1$-orbits on $H,$ in particular, $h^{-1} \in \O(\A),$ and this implies that
$h \in \O(\A)$ as well.

\subsection{Isomorphisms of S-rings}

Let $\A$ be an S-ring over a group $H$ and $\B$ be an S-ring over a group $K$. A bijection $f  : H \to K$ is called an {\em (combinatorial) isomorphism}
between $\A$ and $\B$ if
$$
\big\{ \Cay(H,T)^f : T \in \Bs(\A) \big\}=\big\{ \Cay(K,S) : S \in \Bs(\B) \big\}.
$$
Here $\Cay(H,T)^f$ is the image of the digraph $\Cay(H,T)$ under $f,$
that is, it has vertex set $K$ and arc set
$\{(x^f,y^f) :  x,y \in H \text{ and }
yx^{-1} \in T\}$.

It follows from the definition that $f$ induces a bijection
$f^* : \Bs(\A) \to \Bs(\B)$ defined by $T^{f^*}=S$ for $T \in \Bs(\A)$
exactly when $\Cay(H,T)^f=\Cay(K,S)$.
We say that $f$ is {\em normalized} if $f$ maps the identity element
$1_H$ to the identity element $1_K$. In the special case when $f$ is
an isomorphism from $H$ to $K,$ we  call $f$ a {\em Cayley isomorphism}.
Notice that, when $f$ is a normalized isomorphism
of $\A,$ then $T^{f^*}=T^f$ holds for all $T \in \Bs(\A)$.
It is well-known that the linear map defined by $\un{T} \mapsto \un{T^{f^*}}$ is an algebra isomorphism between the $\Q$-algebras $\A$ and $\B$ (cf. also \cite{HM,MPon}). Using this fact, it is not hard to show that $(ST)^f=S^f T^f$ holds for any normalized isomorphism $f$ of $\A$ and any basic sets $S,T \in \Bs(\A)$.  Some properties are listed below.

\begin{prop}{\rm (\cite[Proposition~2.7]{HM})}\label{HM3}
Let $f : \A \to \B$ be a normalized isomorphism from an
S-ring $\A$ over a group $H$ to an S-ring $\B$ over a group $K,$
and let $E \le H$ be an $\A$-subgroup. Then,
\begin{enumerate}[(i)]
\item the image $E^f$ is a $\B$-subgroup of $K$. Moreover, the restriction
$f_E : E \to E^f$ is an isomorphism between $\A_E$ and $\A_{E^f}$.
\item For each $h\in H,$ $(Eh)^f=E^f h^f$.
\item If $E \trianglelefteq H$ and $E^f \trianglelefteq K,$
then the mapping $f^{H/E} : H/E \to K/E^f,$
defined by $(Eh)^{f^{H/E}}:=E^f h^f$ is a normalized isomorphism between
$\A_{H/E}$ and $\B_{K/E^f}$.
\end{enumerate}
\end{prop}

In this paper, we will be interested exclusively in isomorphisms
between S-rings over the same group $H$. We adopt the notation used in \cite{HM}, and
denote by $\Iso(\A)$ the set of all isomorphisms from $\A$ to S-rings over $H,$ that is,
$$
\Iso(\A)=\{f\in Sym(H) : f \textrm{ is an isomorphism from  $\A$ onto an S-ring over }H  \};
$$
and let $\Iso_1(\A)=\{f \in \Iso(\A) : 1^f=1\}$.

Note that, $\Iso(\A) \subseteq \Sym(H),$  but it is not necessarily
a subgroup. It follows from the definition that for any $\gamma \in \Aut(\A)$ and
$\psi \in \Aut(H),$ their product $\gamma\psi$ is an isomorphism from $\A$ to an S-ring over $H$.
Therefore, $\Aut(\A)\Aut(H) \subseteq \Iso(\A)$. Now, we say that
$\A$ is a {\em CI-S-ring} or simply that $\A$ is CI, if
$\Iso(\A)=\Aut(\A)\Aut(H)$.
This definition was given by Hirasaka and Muzychuk in \cite{HM}, where the
following proposition was proved:

\begin{prop}{\rm (\cite[Theorem~2.6]{HM})}\label{HM4}
Let $G \le \Sym(H)$ be a $2$-closed group with $H_R \le G,$ and
let $\A=V(H,G_1)$. Then the following conditions are equivalent:
\begin{enumerate}[(i)]
\item $G$ is $H_R$-transjugate.
\item $\Iso(\A)=\Aut(\A)\Aut(H)$.
\item $\Iso_1(\A)=\Aut(\A)_1\Aut(H)$.
\end{enumerate}
\end{prop}

Thus, the  CI$^{(2)}$-property for a group $H$ is equivalent to the CI-property
for all Schurian S-rings over $H$.
In the last lemma of this subsection we collect further properties of S-ring isomorphisms.

\begin{lem}\label{L-iso}
Let $\A$ be an S-ring over a group $H,$ and let $f \in \Iso_1(\A)$.
\begin{enumerate}[(i)]
\item If $H$ is abelian and $T^f=T$ for some $T \in \Bs(\A),$
then $(T^{(k)})^f=T^{(k)}$ for any integer $k$ coprime to $|H|$.
\item Let $E \le H$ be an $\A$-subgroup such
that $E \le \rad(T)$ for some $T \in \Bs(A)$. If $(TE)^f=TE$
then $T^f=T$.
\end{enumerate}
\end{lem}

\begin{proof}
(i): Since $T^f=T,$ $f \in \Aut(\Cay(H,T))$. Let us consider the
S-ring $\sgg{T}$. By Eq.~\eqref{Eq-auto}, $f \in \Aut(\sgg{T})$. On the other hand, by Theorem~\ref{T-1st-multi}, $T^{(k)} \in \Bs(\sgg{T}),$ and (i) follows.
\smallskip

(ii): This follows from  $(TE)^f=TE$ and $ET=TE=T$ as  $E \le \rad(T)$.
\end{proof}

\subsection{$p$-S-rings}

We say that an S-ring $\A$ over a group
$H$ is a {\em $p$-S-ring} if $H$ is a $p$-group, and all basic sets $T \in \Bs(\A)$
have $p$-power size, see \cite{HM}. The following proposition follows from results about
$p$-schemes proved in \cite{Z} (see \cite[Theorem~3.3]{HM}).  For sake of
easier reading, we give a proof  using only the definition of an S-ring.

\begin{prop}\label{P-p-S1}
Let $\A$ be a $p$-S-ring over a $p$-group $H$. Then
\begin{enumerate}[(i)]
\item the thin radical $\O(\A)$ is non-trivial;
\item there exists a chain of $\A$-subgroups
$$
H_0=\{1\} < H_1 < \cdots < H_r = H,
$$
such that $|H_{i+1}:H_i|=p$ for all $i \in \{0,\ldots,r-1\}$.
\end{enumerate}
\end{prop}

\begin{proof}
By definition, $\sum_{T\in \Bs(\A), T \ne \{1\}}|T|=|H|-1$. Since all cardinality $|T|$ as
well as $|H|$ are $p$-powers,  (i) follows.

We prove (ii) by induction on $|H|$. The statement is trivial for $|H|=p$. For
the rest of the proof it is assumed that $|H|>p$.
Choose a maximal non-trivial $\A$-subgroup $K < H,$ that is, $H$ is the only
$\A$-subgroup which contains properly $K$. Let $|K|=p^m$.
Let us consider the sets $KTK, \, T\in \Bs(\A)$.
These sets form a partition of $H$ because $T_1 \subseteq KT_2K$
or $T_1 \cap KT_2K = \emptyset$ for any basic sets $T_1,\, T_2 \in
\Bs(\A)$ by Eq.~\eqref{Eq-T/K}, and therefore,  $KT_1K=KT_2K$
or $KT_1K \cap KT_2K =\emptyset$.
Note that,  $p^m$ divides $|KTK|$ for all $T,$
and $KTK=K$ for all basic sets $T \subset K$. Thus,
$|H|=\sum_{T\in \Bs(\A), T \not\subseteq K}|KTK| + |K|,$
and so there exists a basic set
$T_1$ such that $T_1 \not\subseteq K$ and $|KT_1K|=p^m$. Then,
$Kt \subseteq KT_1K$ and $tK \subseteq KT_1K$ for all $t\in T_1$.
This together with $|Kt|=|tK|=|KT_1K|=p^m$ shows that any $t\in T_1$ normalizes
$K$. Thus, $K \trianglelefteq \sg{K,T_1}=H,$ where the latter equality follows by the maximality of $K$. Now, (ii) follows by applying the induction  hypothesis to
the S-rings $\A_K$ and $\A_{H/K}$.
\end{proof}

\begin{prop}{{\rm (\cite[Proposition~3.4(i)]{HM})}}\label{HM5}
Let $\A$ be a $p$-S-ring over an abelian $p$-group $H$.
If there exists a basic set $T \in \Bs(A)$ with $|T|=|H|/p,$ then
$\A$ decomposes to the wreath product $\A=\A_K \wr \A_{H/K},$ where
$K \le H$ is an $\A$-subgroup with index $|H:K|=p$.
\end{prop}

The next result is the classification of all $p$-S-rings over $\Z_p^3$.

\begin{thm}\label{HMSW}
{\rm (\cite{HM,SW})}
Every $p$-S-ring over the group $\Z_p^3$ for an odd prime
$p$ is Cayley isomorphic to one of the S-rings given in Table~1.
\end{thm}

\begin{table}\label{Table}
\begin{tabular}{|c|c|c|c|} \hline
no.  & $p$-S-ring & Schurity & indecomposability \\ \hline
$1$. & $\Q\Z_p^3$ & yes & yes  \\ \hline
$2.$ & $\Q\Z_p^2 \wr \Q\Z_p$ & yes & no \\ \hline
$3.$ & $\Q\Z_p \wr \Q\Z_p^2$ & yes & no  \\ \hline
$4.$ & $(\Q\Z_p \wr \Q\Z_p) \otimes \Q\Z_p$ & yes & no  \\ \hline
$5.$ & $\Q\Z_p \wr \Q\Z_p \wr \Q\Z_p$ & yes & no \\ \hline
$6.$ & $V(\Z_p^3,\sg{x})$,
{\footnotesize $x=\begin{pmatrix} 1&1&0 \\ 0&1&1\\0&0&1 \end{pmatrix}$}
& yes & yes  \\ \hline
\end{tabular}
\caption{$p$-S-rings over $\Z_p^3$ for an odd prime $p$.}
\end{table}

\begin{rem}
Hirasaka and Muzcyhuk~\cite{HM} classified the Schurian S-rings, and it was proved later by Spiga and Wang~\cite{SW}
that all $p$-S-rings over $\Z_p^3$ are Schurian (see \cite[Theorem~1]{SW}).
\end{rem}

Let $H$ be a group isomorphic to $\Z_p^3$ for an odd prime $p$. An S-ring $\A$
over $H$ is called {\em exceptional} if it is Cayley isomorphic to the S-ring in the $6$th row of Table~1, see \cite{HM}. Exceptional S-rings will play an important role in later sections.

\begin{lem}\label{L-exp}
Let $\A$ be an exceptional S-ring over a group $H,\, H \cong \Z_p^3$.  Then $|\Aut(\A)|=p^4$ and $\Iso_1(\A)=\Aut(H)$.
\end{lem}

\begin{proof}
Consider the S-ring in the $6$th-row of Table~1. Denote by $T$ its basic
set  containing the element $(1,0,0) \in \Z_p^3$. Since $p>2,$ it 
follows quickly that
$|T|=p$ and $\sg{T}=\Z_p^3$. This implies that $\A$ has also a basic set
$T'$ such that $|T'|=p$ and $\sg{T'}=H$. The S-ring $\A$ is Schurian.
Thus, $T'$ is  equal to an $\Aut(\A)_1$-orbit, and by
Proposition~\ref{HM2}, $|\Aut(\A)|=p^4$. Thus, $H_R \trianglelefteq \Aut(\A),$ and so
$\Aut(\A)_1 \le \Aut(H)$. Since $H$ is a CI$^{(2)}$-group, $\A$ is a CI-S-ring,
see Proposition~\ref{HM4}, and we can write
$\Iso_1(\A)=\Aut(\A)_1\Aut(H)=\Aut(H)$.
\end{proof}

We finish the subsection with further properties.

\begin{lem}\label{L-p-S2}
Let $\A$ be a $p$-S-ring over a group $H,$ $K\le H$ be an
$\A$-subgroup with index $|H:K|=p,$ and let  $T\in \Bs(\A)$. Then the following hold:
\begin{enumerate}[(i)]
\item $T$ is contained in a $K$-coset. In particular, $\rad(T) \le K$.
\item Let $L \le H$ be an $\A$-subgroup of order $p$ such that
$L \trianglelefteq H$ and $L \not\le \rad(T)$. Then for any $h \in T,$
$hL \cap T=Lh \cap T=\{h\}$.
\item If $H$ is an abelian group and  $|\O(\A) \cap K| |T| > |H|/p,$ then $\O(\A) \cap \rad(T) \ne \{1\}$.
\end{enumerate}
\end{lem}

\begin{proof}
(i): Let us consider the quotient S-ring $\A_{H/K}$.
By Eq.~\eqref{Eq-T/K}, $\A_{H/K}$ is a $p$-S-ring.
Since $H/K \cong \Z_p,$ its only $p$-S-ring is $\Q\,  H/K,$ and so
$\A_{H/K}=\Q \, H/K$.  In particular, $|T/K|=1,$ hence
$T \subseteq K h$ for a coset $K h$ and (i) follows.
\smallskip

(ii) By Eq.~\eqref{Eq-T/K},  there is a positive integer $k$ such that
$|hL \cap T|=|Lh \cap T|=k$ for all $h \in T$.
As $|T|$ is a $p$-power, $k=1$ or $p$. If $k=p,$
then we find $LT=TL=T,$ so $L \le \rad(T),$ which is excluded by one of the assumptions. Thus, $k=1$ and (ii) follows.
\smallskip

(iii): Assume that $|\O(\A) \cap K| |T| > |H|/p$.
Let us consider the sets $eT, \, e \in \O(\A) \cap K$.
By Eq.~\eqref{Eq-eT} and (i), these are all basic sets contained
in a coset $Kh$.  If these are pairwise distinct, then
$|\O(\A) \cap K| |T|=\sum_{e\in \O(\A) \cap K}|eT| \le |Kh|=|H|/p,$ a contradiction.
Thus, $eT=e'T$ for distinct $e,e' \in \O(\A) \cap K$.
Using this and that $H$ is abelian, we find $e^{-1}e'T=Te^{-1}e'=T,$ and so
$e^{-1}e' \in \O(\A) \cap \rad(T),$ by which (iii) follows.
\end{proof}

\section{On CI-S-rings over $\Z_p^n$}

In this section we give three propositions about CI properties of S-rings over the groups $\Z_p^n$. The first one is a necessary condition for an S-ring to be non-CI. It is essentially contained in the proof of \cite[Proposition~3.9]{HM}.

\begin{prop}\label{P-regular-R}
Suppose that $\A=V(H,P_1)$ is a non-CI-S-ring, where
$H \cong \Z_p^n$  and $H_R \le P \le \Sym(H)$ is a $p$-group.
Then
the normalizer $N_{\Aut(\A)}(H_R)$ contains a subgroup
$K$ for which the following hold:
\begin{enumerate}[(i)]
\item $K \cong \Z_p^n,$ it is regular on $H,$ and $K \ne H_R$.
\item The stabilizer $(KH_R)_1$ is elementary abelian, and
$$
|C_{H_R}((KH_R)_1)| \ge |K \cap H_R|.
$$
\end{enumerate}
\end{prop}

\begin{proof}
Let $G=\Aut(\A)$ and $N=N_{\Aut(\A)}(H_R)$.
Note that, since $G=P^{(2)}$ and $P$ is a
$p$-group, $G$ is a $p$-group as well, see Proposition~\ref{W}.
We first show the existence of a subgroup $K \le N$ that has
all properties given in (i). If $G=N,$ then the existence
of the required subgroup $K$ follows
from the condition that $\A$ is a non-CI-S-ring. Now, suppose
that  $N < G$.
Then, since $G$ is a $p$-group, the normalizer
$N_G(N)\ne N,$  hence we may choose
some $g \in N_G(N) \setminus N$. We let $K=(H_R)^g$. It
is straightforward to see that $K$ has all properties given in (i).

Now, we turn to part (ii).  Consider the group $Q=K H_R$.
Clearly, $Q=Q_1 K=Q_1 H_R$ and $Q_1 \cong Q/H_R \cong K/(K \cap H_R)$. Thus, $Q_1$ is an elementary abelian group. Also,
as both $K$ and $H_R$ are abelian, $K \cap H_R \le Z(Q),$ implying
that $|C_{H_R}(Q_1)| \ge |K \cap H_R|$.  This completes the proof of part (ii).
\end{proof}

The second proposition will be a sufficient condition for an S-ring over $\Z_p^n$ to be CI.

\begin{defi}\label{D-teq}
We say that an S-ring $\A$ of a group $H$ is {\em $\teq$-minimal} if
$$
\{ X \le  \Sym(H) :  X \ge H_R \text{ and } X \teq \Aut(\A) \}=\{ \Aut(\A) \}.
$$
\end{defi}

For example, the full group algebra $\Q H$ is a $\teq$-minimal S-ring.
The obvious examples for non-$\teq$-minimal S-rings are the S-rings of rank $2$ (the two basic sets are $\{1\}$ and $H \setminus \{1\}$). Clearly, $\Aut(\A)=\Sym(H),$ and
thus $\Aut(\A) \teq X$ whenever $H_R \le X \le \Sym(H)$  and $X$ is $2$-transitive on $H$.

\begin{prop}\label{P-teq}
Let $\A$ be a Schurian $p$-S-ring over a group $H \cong \Z_p^n,$ and $K \le H$ be an $\A$-subgroup of order $p$ such that $\A_{H/K}$ is a $\teq$-minimal CI-S-ring. Then $\A$ is a CI-S-ring.
\end{prop}

\begin{proof}
Let $G=\Aut(\A)$ and choose $L \le G$ such
that $L$ is regular on $H$ and $L \cong H$. Because of Proposition~\ref{HM4} it is enough to show that $L$ and $H_R$ are conjugate  in $G$.

We write $\bar{G}=G^{H/K}, \, \bar{H}_R=(H_R)^{H/K}$ and
$\bar{L}=L^{H/K}$. Note that $\bar{H}_R=(H/K)_R$.
The group $\bar{L}$ is abelian acting transitively on
$H/K$. It follows that it is regular, and $\bar{L}\cong \Z_p^{n-1}$.
Since $\A_{H/K}$ is a CI-S-ring,
$\bar{L}=(\bar{H}_R)^x$ for some $x\in \Aut(\A_{H/K})$.
For sake of simplicity we denote by $\bar{1}$ the identity of $H/K$.

We claim that $\Aut(\A_{H/K})=\bar{G}$. To settle this it is
sufficient to show that $\Aut(\A_{H/K}) \teq \bar{G},$ and
use the assumption that $\A_{H/K}$ is $\teq$-minimal.
We have to show that $\A_{H/K}=V(H/K,\bar{G}_{\bar{1}})$.
Here we copy the proof of \cite[Proposition~2.8(ii)]{HM}.
Since $K$ is an $\A$-subgroup, $K^g=K$ for
any $g \in G_1,$ and thus by Proposition~\ref{HM3}(ii), the coset $Kh$ is mapped by $g^{H/K}$ to $(Kh)^{g^{H/K}}=Kh^g$. A basic set of $\A_{H/K}$ is in the form $T/K$ where
$T \in \Bs(\A)$. Since $\A=V(H,G_1),$
we find that $T=h^{G_1}$ for some $h\in H$. Observe that $G_{\{K\}}=K_RG_1$, and for any $g \in G,$
$g^{H/K} \in \bar{G}_{\bar{1}}$ if and only if $g$ fixes setwise the subgroup $K$. This implies that $g=k_R g'$ for some $k\in K$ and $g' \in G_1$. Now, we can express the
$\bar{G}_{\bar{1}}$-orbit of $Kh$ as
\begin{eqnarray*}
(Kh)^{\bar{G}_{\bar{1}}}&=&
\{ (Kh)^{x^{H/K}} : x=k_R g, \, k\in K, \, g \in G_1 \}=
\{K (kh)^g : k \in K, \, g \in G_1\} \\
&=& \{K h^g : g \in G_1\}=h^{G_1}/K=T/K.
\end{eqnarray*}
We conclude that $T/K$ is an orbit of $\bar{G}_{\bar{1}},$
and the claim follows.

Recall that $\bar{L}=(\bar{H}_R)^x$ for some $x\in \Aut(\A_{H/K})=\bar{G}$.
Choose $g\in G$ such that $g^{H/K}=x^{-1}$. Then $L^g \le G_{H/K}H_R,$ where $G_{H/K}$ denotes the kernel of $G$ acting on $H/K$. Write $M=G_{H/K}H_R$.
Since both $G$ and $(H_R)^{H/K}$ are
$2$-closed groups, it follows by Proposition~\ref{HM1}(ii) that
$M$ is also $2$-closed. We are done if we show that
$M$ is $H_R$-transjugate, because then
$(L^g)^{g'} =H_R$ for some $g' \in M,$ showing that $L$ and $H_R$ are indeed conjugate in $G$.

Again, because of Proposition~\ref{HM4} we are done if we show that
the S-ring $\B=V(H,M_1)$ is CI. Then $\A \subseteq \B,$ and thus $K$ is also a $\B$-subgroup. Note that $G_{H/K}\cap H_R=K_R$ and $M=M_1H_R$. Then $|M_1|=|M|/|H_R|=|G_{H/K}H_R|/|H_R|=|G_{H/K}|/|K_R|$ and $|G_{H/K}|=|(G_{H/K})_1||K_R|$. It follows $|M_1|=|(G_{H/K})_1|$ and  hence $M_1=(G_{H/K})_1$. Since $(G_{H/K})_1\lhd G_1$ and all orbits of $G_{H/K}$ are the cosets of $K$ in $H$ which have order $p$,  we have $K \le  \O(\B)$ and
all basic sets of $\B=V(H,M_1)$ not contained in $\O(\B)$ are $K$-cosets.

Let $f \in \Iso_1(\B)$. In order to prove that $\B$ is a CI-S-ring, we have to find an automorphism $\varphi \in \Aut(H)$ such that
\begin{equation}\label{Eq-Tf}
T^{f \varphi}=T \text{ for all } T \in \Bs(\B).
\end{equation}

Choose a minimal generating set $\{h_1,\ldots,h_n\}$ of $H$ such that  $\{h_1,\ldots,h_\ell\}$, $\ell\leq n$, is a generating set of $\O(\B)$ with $h_1\in K$. By Proposition~\ref{HM3}, $K^f\leq H$ and $(Kh_i)^f=K^fh_i^f$. Since $1^f=1,$
$T^f \in \Bs(\B^f)$ for every basic set $T \in \Bs(\A)$.
Using this and that each $Kh_i$ is a $\B$-subset, we find that each  $K^fh_i^f$ is a $\B^f$-subset, and so
$\sg{K^fh_1^f,\ldots,K^fh_n^f} \le H$ is a $\B^f$-subgroup.
By Proposition~\ref{HM3}(i),
$|\sg{K^fh_1^f,\ldots,K^fh_n^f}|=
|\sg{K^fh_1^f,\ldots,K^fh_n^f}^{f-1}|$.
Thus,  $H=\sg{K^fh_1^f,\ldots,K^fh_n^f}$.
Since $h_1\in K,$ it follows that $K^fh_1^f=K^f,$ and
$\{h_1^f,\cdots,h^f_n\}$ is also a minimal generating set of $H$. Define $\varphi$ as the induced automorphism of $H$ by
$\varphi: h_i^f\mapsto h_i$  for $1\leq i\leq n$. Then $h_i^{f \varphi}=h_i$. To finish the proof, it suffices to show that Eq.~\eqref{Eq-Tf} holds.

Set $f_1=f \varphi$. Clearly, $f_1\in \Iso_1(\B)$. Recall that for any $S,T\in \Bs(\B)$, $(ST)^{f_1}=S^{f_1}T^{f_1}$ (see the paragraph preceding Proposition~\ref{HM3}). Then $f_1$ fixes each element in $\O(\B)$ because $f_1$ fixes a generating set of $\O(\B)$.
In particular, $K^{f_1}=K$ as $K \leq \O(\B),$ and
Eq.~\eqref{Eq-Tf} holds whenever $T \subset \O(\B)$.
Now, suppose that $T \not\subset \O(\B)$.
Let us consider the isomorphism $f_1^{H/K}$ of $\B_{H/K}$ induced by $f_1$
(for the definition of $f_1^{H/K},$ see Proposition~\ref{HM3}(iii)).
The quotient S-ring $\B_{H/K}=\Q\, H/K$. Since
$\Q\, H/K$ is a CI-S-ring and $\Aut(\Q \, H/K)=(H/K)_R$, it follows that $\Iso_1(\B_{H/K})=
\Aut(\Q\, H/K)_1\Aut(H/K)=\Aut(H/K)$. Also, $f_1^{H/K}
\in \Iso_1(\B_{H/K}),$ because $K^{f_1}=K,$ and so
$f_1^{H/K} \in \Aut(H/K)$.
On the other hand, as $f_1$ fixes all generators $h_i,$
$f_1^{H/K}$ fixes a generating set of $H/K,$ and so $f_1^{H/K}$ is the identity mapping.
Since $T \not\subset \O(\B),$
$T=Kh$ for some $h \in H \setminus K,$
and we can write
$T^{f_1}=(Kh)^{f_1}=(Kh)^{f_1^{H/K}}=Kh=T$.
\end{proof}

Proposition~\ref{P-teq} will be especially useful in conjunction with the fact that
all indecomposable Schurian $p$-S-rings over $\Z_p^4$ are $\teq$-minimal.
We prove the latter fact in Section~4.
\medskip

Recall that,  the  CI$^{(2)}$-property for a group $H$ is equivalent to the CI-property
for all Schurian S-rings over $H$ (see Proposition~\ref{HM4}).
The third  proposition is the following refinement:

\begin{prop}\label{P-iff}
Let $H$ be a group isomorphic to $\Z_p^n$ for an
odd prime $p$.
Then the following conditions are equivalent:
\begin{enumerate}[(i)]
\item $H$ is a CI$^{(2)}$-group.
\item All S-rings $V(H,A)$ are CI-S-rings where
$A < \Aut(H)$ is a $p$-group with $|C_{H_R}(A)| \ge p^2$.
\end{enumerate}
\end{prop}

\begin{proof}
Notice that, the implication (i) $\Rightarrow$ (ii) is a
direct consequence of Proposition~\ref{HM4}.

Now, we turn to the implication (ii) $\Rightarrow$ (i).
Let $G \le \Sym(H)$ be a $2$-closed subgroup with $H_R \le G,$ and let $K \le G$ be a regular subgroup such that $K \cong H$. We have to show that $K$ and $H_R$ are conjugate in $G$.

Now, choose a Sylow $p$-subgroup $P$ of $G$ such that $H_R \le P$. Since $G$
is $2$-closed, by Proposition~\ref{W}, $P$ is $2$-closed, that is, $P^{(2)}=P$.
By Sylow Theorem, $K^x \le P$ for some $x \in G,$ hence we may
assume that $K \le P$. According to Proposition~\ref{M} there exists some $y \in \sg{H_R,K}^{(2)}\leq P^{(2)}=P$ such that
$|C_{H_R}(K^y)| \ge p^2$. Let $Q=\sg{H_R,K^y}$.
Then $Q^{(2)} \le P^{(2)}=P,$ and $Q^{(2)}$ is also a $p$-group.
It is sufficient to show that $Q^{(2)}$ is $H_R$-transjugate.

Let us consider the normalizer $N=N_{Q^{(2)}}(H_R)$. Since $|C_{H_R}(K^y)| \ge p^2,$ it follows that $|C_{H_R}(Q)| \ge p^2$. By Proposition~\ref{P-center-G2},  $C_{H_R}(Q)=H_R \cap Z(Q) \leq H_R \cap Z(Q^{(2)})=C_{H_R}(Q^{(2)})$.
Therefore, $|C_{H_R}(Q^{(2)})| \ge p^2,$ and as
$N^{(2)} \le Q^{(2)},$ it follows that
$|C_{H_R}(N^{(2)}_1)| \ge p^2$.
By the hypothesis in (ii), the S-ring $V(H,N^{(2)}_1)$ is a CI-S-ring. Equivalently,
$N^{(2)}$ is $H_R$-transjugate.

We finish the proof by showing that
$N^{(2)}=Q^{(2)}$. In doing this we use the same idea as in
the proof of \cite[Proposition~1]{Sp9}. Assume to the contrary
that $N^{(2)} < Q^{(2)}$. Since $Q^{(2)}$ is a $p$-group,
we can choose an element $z \in N_{Q^{(2)}}(N^{(2)}) \setminus N^{(2)}$. Then $(H_R)^z \le N^{(2)}$ . Since $N^{(2)}$ is $H_R$-transjugate, $(H_R)^z=(H_R)^{z'}$ for some $z' \in N^{(2)},$ and so we find $z'z^{-1} \in N_{Q^{(2)}}(H_R)=N,$  from which
$z \in N^{(2)},$ a contradiction. Therefore, $Q^{(2)}=N^{(2)},$
showing that $Q^{(2)}$ is $H_R$-transjugate, as required.
\end{proof}

In fact, we are going to derive Theorem~\ref{2} by showing that the condition in case (ii) of Proposition~\ref{P-iff} holds when $H \cong \Z_p^5$.

\begin{thm}\label{3}
Let $H \cong \Z_p^5$ for an odd prime $p$. Then all S-rings
$V(H,A)$ are CI-S-rings where $A < \Aut(H)$ is a $p$-group with
$|C_{H_R}(A)| \ge p^2$.
\end{thm}

The proof of Theorem~\ref{3} will be given in Sections~5 and 6.

\section{Indecomposable Schurian $p$-S-rings over $\Z_p^n$ are
$\teq$-minimal for $n \le 4$}

We set some notation  that will be used throughout the rest of the paper:
\medskip

\noindent{\bf Notation}. \
From now on $p$ will stand for an odd prime, and
$H$ will denote a group isomorphic to $\Z_p^n$. The group $H$ will be regarded as the additive group of an $n$-dimensional vector space over the field $GF(p)$.
The elements of $H$ will be denoted by lower case letters $u, \, v,$ etc., while the subgroups of $H$ by upper case letters $U,\, V,$ etc. As usual, the identity element will be denoted by $0,$ and
the inverse of an element $u \in H$ by $-u$.
For an integer $k$ and a subset $T\subseteq H$ we write $T^{(k)}=kT=\{ ku : h \in T\},$ where $ku=u+\cdots+u,$ with $|k|$ summands if $k >0,$ and $ku=-(u+\cdots+u)$ otherwise.
\medskip

It turns out that all indecomposable $p$-S-rings over the group
$\Z_p^n$ are $\teq$-minimal for any odd prime $p$ and $n \le 3$. This is not hard to see for the groups $\Z_p$ and $\Z_p^2$. The full group algebra $\Q \Z_p$ is the only $p$-S-ring over $\Z_p;$ and up to Cayley isomorphisms, there are two $p$-S-rings over
$\Z_p^2:$ $\Q \Z_p^2$ and $\Q \Z_p \wr \Q \Z_p,$ and the latter one is decomposable. Theorem~\ref{HMSW} shows that, up to Cayley isomorphisms, there are two indecomposable $p$-S-rings over $\Z_p^3$: $\Q \Z_p^3$ and
the exceptional $p$-S-ring given in the $6$th row of Table~1. 
By Lemma~\ref{L-exp}, the automorphism group of an exceptional $p$-S-ring has  order $p^4,$ hence it is $\teq$-minimal. In this section, we extend this result to the Schurian indecomposable $p$-S-rings over $\Z_p^4$.

\begin{thm}\label{T-teq-min}
All indecomposable Schurian $p$-S-rings over the group
$\Z_p^4$ are $\teq$-minimal for any odd prime $p$.
\end{thm}

The proof of the theorem will be given in the end of the section following three preparatory lemmas.

Recall that, if $\A$ is an S-ring over $H$ and $W \le H$ is an $\A$-subgroup,  then the $W$-cosets in $H$ form a block system for  $\Aut(H)$.

\begin{lem}\label{L-kernel}
Let $\A$ be an indecomposable S-ring over a group
$H \cong \Z_p^n,$ and let $W$ be an $\A$-subgroup with
$|W|=p$. Then the kernel
$$
\Aut(\A)_\delta = W_R=\{ w_R : w \in W\},
$$
where $\delta$ denotes the block system $H/W$.
\end{lem}

\begin{proof}
Let $K=\Aut(\A)_\delta$. It is clear that
$W_R \le K,$ and thus it is enough to prove that the stabilizer
$K_0$ is trivial. By Lemma~\ref{L-p-S2}(ii), for every basic set $T \in \Bs(\A),$
\begin{equation}\label{Eq-coset}
W \not\le \rad(T) \Rightarrow T \cap (W+u)=\{u\} \text{ for all } u\in T.
\end{equation}

We define recursively a finite sequence $T_1,\ldots,T_r$ of basic sets of $\A$ as follows. Let $T_1=\{w\}$ where $w$ is an arbitrary nonzero element in $W$.
Now, suppose that the sets $T_1,\ldots,T_i$ are already defined for $i \ge 1$.
If $H=\sg{T_1 \cup \cdots \cup T_i},$ then finish the procedure and let $r=i$. Otherwise,
choose $T_{i+1}$ to be a basic set in
$H \setminus \sg{T_1 \cup \cdots \cup T_i}$ such that
$W \not\le \rad(T_{i+1})$. Notice that, such $T_{i+1}$ does exist because $\A$ is indecomposable.

Let $S=T_1 \cup \cdots \cup T_r$ and consider the Cayley graph
$\Cay(H,S)$.  Note that, $\Aut(\A) \le \Aut(\Cay(H,S))$.
It is clear from the construction that $\sg{S}=H,$ hence
$\Cay(H,S)$ is connected. We claim that $S$ has the property that,  whenever a $W$-coset intersects $S$, it does intersect it at exactly one element.  Suppose to the contrary that there exist $u_1, u_2 \in S$ such that
$u_1 \ne u_2$ and $u_1-u_2 \in W$.
Then $u_1 \in T_i$ and $u_2 \in T_j$ for some $i,j \in \{1,\ldots,r\}$. It follows  from the construction of $S$ and Eq.~\eqref{Eq-coset}
that $i \ne j$, and we may assume w.\ l.\ o.\ g.\ that $i < j$. Thus, $u_2 \in \sg{T_i,W} \le  \sg{T_1,\ldots,T_i},$ and so $u_2 \in \sg{T_1 \cup \cdots \cup T_{j-1}}  \cap T_j,$  a contradiction.
Now, using that $\Aut(\A) \le \Aut(\Cay(V,S))$ and the above property of $S,$ we find that every element in $K_0$ fixes
all neighbors of $0$ in $\Cay(H,S)$. This and the connectedness of $\Cay(H,S)$ yield that $|K_0|=1$.
\end{proof}

For $x\in \Aut(H),$ we define  $C_H(x)=\{u\in H : u_R \in C_{H_R}(x)\}$.

\begin{lem}\label{L-center1}
Let $\A$ be an indecomposable S-ring over a group
$H \cong \Z_p^4$ and let $x \in N_{\Aut(\A)}(H_R)$
such that $x \ne \id_H$ and $0^x=0$. Then $|C_H(x)| \le p^2$.
\end{lem}

\begin{proof}
Since $x \ne \id_H,$ it follows that $|C_H(x)| \le  p^3$.
Assume to the contrary that $|C_H(x)|=p^3$.
Let $U=C_H(x),$ and for a fixed $v_1 \in H \setminus U,$ let $W = \sg{v_1^x-v_1}$. Then $|W|=p,$ and it follows that the orbit $v^{\sg{x}} = W+v$ for all $v \in H \setminus U$. Observe that, $W$ is not an $\A$-subgroup. For otherwise,
$x$ belongs to the kernel of $\Aut(\A)$ acting on $H/W,$ which is impossible
by Lemma~\ref{L-kernel}.

Let $U'$ be an $\A$-subgroup of order $p^3$. If $U=U',$ then  define
$V = \bigcap_{T \in \Bs(\A), \, T \subseteq H \setminus U} \rad(T)$.
Clearly, $V$ is an $\A$-subgroup such that $W \le V \le U,$ and
$\A$ is an $U/V$-wreath product, a contradiction.
Hence, $U' \ne U,$ and in particular, $U$ is not an $\A$-subgroup.

Let $T_1 \in \Bs(\A)$ such that $T_1 \subset U'$ and 
$T \not\subset U$. Then $|T_1| \ge p,$ and since $W$ is not an 
$\A$-subgroup, it follows that 
$|T_1|=p^2$. By Proposition~\ref{HM5}, $T_1$ is equal to a $U''$-coset for some $\A$-subgroup $U''$ such that $|U''|=p^2$ and 
$W< U''$.  We find
$U''=W+W'$ for an $\A$-subgroup $W'$ of order $p$.
Since $W' \le \O(\A),$ $W' < U,$ and it follows that $U''=U' \cap U$.

Now, choose $T \in \Bs(\A)$ such that $T \not\subset U \cup U'$ and
$T \cap U \ne \emptyset$.
Notice that such $T$ exists because $U$ is not an $\A$-subgroup.
It follows that $|T| > p$. Fix an element $v \in T \cap U$.
Then $T \subseteq U'+v,$ see Lemma~\ref{L-p-S2}(i).
Since $\A$ is indecomposable, it follows from
Proposition~\ref{HM5} that $|T|=p^2$. Let $v' \in T \setminus U$.
If $W' \le \rad(T),$ then we find $U''+v'=W'+(W+v') \subseteq W'+T=T$. Thus, $T=U''+v',$ contradicting that
$T \cap U \ne \emptyset$. Thus, $W' \not\le \rad(T),$ and by
Eq.~\eqref{Eq-T/K}, every $W'$-coset intersects $T$ in at most $1$ element . Consequently, any $U''$-coset intersects $T$ in at most $p$ elements. The $U''$-cosets contained in $U'+v$ can be listed as $U''+ku+v,$ where
$k\in \{0,1,\ldots,p-1\}$ and $u$ is any fixed element in $U' \setminus U$.
Let $T_i=T \cap (U''+iu+v), \, i \in \{0,1,\ldots,p-1\}$.
It follows that the sets $T_i$ form a partition of $T,$ $T_i$ is
a $W$-coset for all $i>0,$ and $|T_0|=p$.

Let us consider the product $\un{T} \cdot \un{(-T)}$ in $\Q H$. Writing
$\un{T} \cdot \un{(-T)}=\sum_{u\in H}a_u u,$ it follows quickly from the above
description of the sets $T_i$ that
$\sum_{u\in U'' \setminus \{0\}}a_u=p^{3}-p^2$.
On the other hand, $\un{T} \cdot \un{(-T)} \in \A,$ and it can be expressed as
the linear combination $\un{T} \cdot \un{(-T)}=\sum_{T' \in \Bs(\A)}b_{T'}\un{T'}$.
Let $w \in W, \, w \ne 0$. Since $W$ is not an $\A$-subgroup, it
follows that the coset
$W'+w$ is a basic set of $\A$. Let us denote the latter basic set
by $T(w)$. It also follows
from description of the sets $T_i$ that $a_w \ge p^2-p$. Thus,
$b_{T(w)} \ge p^2-p$ as well, and as $w$ was chosen arbitrarily from
$W \setminus \{0\},$ we arrive at a contradiction as follows:
$$
p^3-p^2=\sum_{u\in U'' \setminus \{0\}}a_u \ge
\sum_{w\in W\setminus \{0\}}b_{T(w)}|T(w)| \ge (p-1)(p^3-p^2).
$$
\end{proof}

Let $\A$ be a $p$-S-ring over $H$. In what follows, we call an ordered $n$-tuple $(v_1,\ldots,v_n)$ of generators of $H$ an {\em $\A$-basis} if all
subgroups in the chain below are $\A$-subgroups
$$
\{0\} < \sg{v_n} < \sg{v_{n-1},v_n} < \cdots < \sg{v_1,\ldots,v_n}=H.
$$
Notice that, if $x \in \Aut(A)$ normalizes $H_R$ and $0^x=0,$ then
$x \in \Aut(H),$ and it can be written in an $\A$-basis as an upper triangular matrix having $1$'s in the diagonal.

\begin{lem}\label{L-normal}
Let $\A$ be an indecomposable $p$-S-ring over a group $H \cong \Z_p^4$. Then
\begin{enumerate}[(i)]
\item $|N_{\Aut(\A)}(H_R|\le p^6;$
\item $H_R$ is normal in $\Aut(\A)$.
\end{enumerate}
\end{lem}

\begin{proof}
Let $G=\Aut(\A)$ and $N=N_{\Aut(\A)}(H_R)$.

(i): Assume to the contrary that $|N| > p^6,$ that is, for
the stabilizer $N_0$ we have $|N_0|> p^2$. Let us fix an $\A$-basis $(v_1,v_2,v_3,v_4)$. This means that $\sg{v_4}$ is an $\A$-subgroup, and we can consider the action of $N_0$ on $H/\sg{v_4}$. By Lemma~\ref{L-kernel}, the latter action is faithful, and hence $N_0$ is isomorphic to a subgroup
of the group of all upper triangular $3 \times 3$ matrices
with each diagonal element equal to $1$.
Therefore, $|N_0|=p^3,$ and we can choose
$x \in Z(N_0)$ that can be written in the basis $(v_1,v_2,v_3,v_4)$ as
$$
x=
\begin{pmatrix}
1 & 0 & 1 &  a \\
0 & 1 & 0  & b \\
0 & 0 & 1 & c \\
0 & 0 & 0 & 1
\end{pmatrix}.
$$

Then, the orbit $v_1^{N_0}$  has size at most $p^2$. This
follows from Proposition~\ref{HM5} and the fact that
$\A$ is indecomposable. Therefore, there exists $y \in N_0$
such that $y \ne \id_H$ and $y$ fixes $v_1$. Using also that $[x,y]=1,$ we find that $(v_1^x)^y=(v_1^y)^x=v_1^x,$
hence $v_1^x=v_1+v_3+av_4 \in C_H(y)$. It follows
that each of $v_1,\, v_3$ and $v_4$ is in $C_H(y)$.
This, however,  contradicts Lemma~\ref{L-center1}.
\smallskip

(ii):  We have to show that $G=N$. Assume to the contrary that
$G > N$. Then $N_G(N) > N$. Choose $g \in N_G(N) \setminus N$ and let $P=(H_R)^g H_R$. Since
$P_0 \le N_0,$ $|P_0| \le p^2$.
If $|P_0|=p,$ then $|P| = |H_R| \cdot |P_0| = p^5,$ hence
$|(H_R)^g \cap H_R|=p^3$.
Then every $x \in P_0$  satisfies $|C_{H_R}(x)| \ge  |(H_R)^g \cap H_R|=p^3,$ a contradiction to Lemma~\ref{L-center1}.
Therefore, $|P_0|=p^2,$ $P_0=N_0$ and each $z \in N_0$ satisfies
$|C_{H_R}(z)| \ge |(H_R)^g \cap H_R|=p^2$.
By Lemma~\ref{L-center1}, $C_{H_R}(z)=(H_R)^g \cap H_R$ whenever $z \ne \id_H$.
Therefore, letting $U=\{ u \in H : u_R \in H_R^g \cap H_R\},$ we can write
\begin{equation}\label{Eq-CHz}
C_H(z)=U \text{ for for all } z\in N_0, \, z \ne \id_H.
\end{equation}

Let us consider the S-ring $\B=V(H,N_0)$.
Clearly, $U \le \O(\B)$. Fix a $\B$-subgroup $V$ such that $V$ has order $p^3$ and
$U < V$. Let $v \in  H \setminus V$ and $T \in \Bs(\B)$ be a basic set such that
$v \in T$. By Lemma~\ref{L-p-S2}(i), $v^z-v \in V$.
Suppose that $v^z-v \in U$  for all $z \in N_0$. This implies that
$T \subseteq U+v,$ and thus either $T=U+v,$ or $|T| \le p$.
In the latter case, however, it follows that $N_0$ contains a nontrivial element $z$ fixing some $v \in T,$ and hence  $C_H(z) \ge \sg{U,v} > U,$ which  contradicts
Eq.~\eqref{Eq-CHz}.  Observe that, if $T=U+v,$ then it is also a 
basic set of $\A$. For otherwise, $\A$ would have a basic set of 
size $p^3,$ contradicting that $\A$ is indecomposable (see Proposition~\ref{HM5}). Now, since $\A$ is not a $V/U$-wreath product, there  exists $v_1 \in H \setminus V$ and $x \in N_0$ for which
$v_1^x-v_1 \notin U$.

Now, define the elements $v_2=v_1^x-v_1,$ $v_3=v_2^x-v_2,$ and let $v_4 \in U$ be an element such that $U=\sg{v_3,v_4}$. It follows that $(v_1,v_2,v_3,v_4)$ is a $\B$-basis, in which
$$
x=
\begin{pmatrix}
1 & 1 & 0 &  0 \\
0 & 1 & 1  & 0 \\
0 & 0 & 1 & 0 \\
0 & 0 & 0 & 1
\end{pmatrix}.
$$

Let $y \in N_0$ such that $N_0=\sg{x,y}$.
By Eq.~\eqref{Eq-CHz}, $C_H(y)=U=\sg{v_3,v_4},$ and thus $y$ can be written in the basis $(v_1,v_2,v_3,v_4)$ in the form
$$
y=
\begin{pmatrix}
1 & a & b &  c \\
0 & 1 & d  & e \\
0 & 0 & 1 & 0 \\
0 & 0 & 0 & 1
\end{pmatrix}.
$$

Using also that $[x,y]=1,$ we find $d=a$ and $e=0$. Thus,
$v_2^y=v_2+dv_3$. On the other hand, $v_2^{(x^d)}=v_2+dv_3$ also holds, and we find $v_2 \in C_H(x^dy^{-1}),$ where $x^dy^{-1} \ne \id_H$.
This contradicts Eq.~\eqref{Eq-CHz}.
\end{proof}

All Schurian $p$-S-rings over $\Z_p^4$ are CI \cite[Theorem~3.1]{HM}. Combining this with Lemma~\ref{L-normal} yields the following
corollary (see also the proof of Lemma~\ref{L-exp}):

\begin{cor}\label{C-normal}
Let $\A$ be an indecomposable Schurian $p$-S-ring over a group
$H, \, H \cong \Z_p^4$ for an odd prime $p$. Then
$\Iso_0(\A)=\Aut(H)$.
\end{cor}

Everything is prepared to prove the main result of this section.

\begin{proof}[Proof of Theorem~\ref{T-teq-min}]
Assume to the contrary that $\A$ is a Schurian indecomposable  $p$-S-ring over
$H,$ which is not $\teq$-minimal. Let $G=\Aut(\A)$. By Lemma~\ref{L-normal},
$H_R \trianglelefteq G$ and $|G|\le p^6$.
As $\A$ is not $\teq$-minimal, $|G|=p^6,$ and there exists
$x \in G_0$ such that $x$ has order $p,$ and  $\A=V(H,\sg{x})$.
In other words, $G \teq K$ where $K=\sg{H_R,x}$. Note that
$G=K^{(2)}$. Let $u \in C_H(x)$. Then $u_R \in Z(K),$ and by
Proposition~\ref{P-center-G2}, $u_R \in Z(K^{(2)})=Z(G),$ implying
that $u \in C_H(y)$ for any $y \in G_0$. We obtain that
$C_H(x) \le C_H(y)$ for all $y \in G_0$.
Also, as $\A=V(H,\sg{x}),$ every basic set of $\A$ has size at most $p$.
Thus, if $|C_H(x)|=p^2,$ then one can find $y\in G_0$ such that $y \ne \id_H$ and $|C_H(y)|=p^3,$ which is impossible by
Lemma~\ref{L-center1}. It remains to consider the case
when $|C_H(x)|=p$. Equivalently, $\rank(x-I)=3,$ and this implies
that, in a suitable basis, denoted by $(v_1,v_2,v_3,v_4),$ $x$ has the following Jordan normal form:
$$
x=
\begin{pmatrix}
1 & 1 & 0 & 0 \\ 0 & 1 & 1  & 0 \\ 0 & 0 & 1 & 1 \\ 0 & 0 & 0 & 1
\end{pmatrix}.
$$
Since $x$ has order $p,$ it follows that $p> 3$.
Let $T$ be the orbit of $v_1$ under $\sg{x}$ (hence under $G_0$). It is not hard to check that $|T|=p$ and $\sg{T}=H$. Then, by Proposition~\ref{HM2},
$|G_0|=p,$ a contradiction. This completes the proof of the theorem.
\end{proof}

We finish the section with a corollary of
Proposition~\ref{P-teq} and Theorem~\ref{T-teq-min}, which will be used several times in the next two sections.

\begin{cor}\label{C-teq-min}
Let $\A$ be a Schurian $p$-S-ring over the group
$H\cong \Z_p^5,$ and let $W$ be an $\A$-subgroup of order $p$. If $\A$ is a non-CI-S-ring, then the S-ring $\A_{H/W}$ is decomposable.
\end{cor}

\section{Proof of Theorem~\ref{3} I: The decomposable S-rings}

We record all assumptions of Theorem~\ref{3} in the following  hypothesis:

\begin{hyp}\label{hyp}
$\A=V(H,A)$ is an S-ring over a group $H\cong \Z_p^5$ for
some odd prime $p,$ and for some subgroup
$A \le \Aut(H)$ with $|C_H(A)| \ge p^2$.
\end{hyp}

Our eventual goal is to show that, assuming Hypothesis~\ref{hyp},
the S-ring $\A$ is CI. In this section, we deal with the particular case when
$\A$ is decomposable.

\begin{thm}\label{31}
Assuming Hypothesis~\ref{hyp},
suppose that $\A$ is decomposable.
Then $\A$ is a CI-S-ring.
\end{thm}

The theorem will be proved in the end of the section following four preparatory lemmas.
In the next three lemmas we study the S-ring $\A$ described in Theorem~\ref{31} which satisfies
additional conditions.

\begin{lem}\label{L-bfly1}
Assuming Hypothesis~\ref{hyp}, suppose that there exist  $\A$-subgroups $U_1,\, U_2, \, W_1$ and $W_2$ with
$|U_1|=|U_2|=p^4, \, |W_1|=|W_2|=p, \, U_1 \ne U_2, \, W_1 \ne W_2$ and $W_1+W_2 < U_1 \cap U_2,$ and also that the following hold:
\begin{enumerate}[(1)]
\item  $\A$ is both a $U_1/W_2$- and a $U_2/W_1$-wreath product.
\item  Both $\A_{U_1/W_1}$ and  $\A_{U_2/W_2}$ are indecomposable.
\item $|A_v| \ne 1$ for some $v \in H \setminus U_1 \cup U_2$.
\end{enumerate}
Then one of the following possibilities holds:
\begin{enumerate}[(i)]
\item $|N_{\Aut(\A)}(H_R)|=p^8,$ and there exists an $\A$-subgroup
$U_3$ such that $|U_3|=p^4, \, U_3 \ne U_i$ for $i \in \{1,2\},$ and
every basic set of $\A$ not contained in $U_1 \cup U_2 \cup U_3$ is equal to a $(U_1 \cap U_2)$-coset.
\item $|N_{\Aut(\A)}(H_R)|=p^9,$ and every basic set of $\A$ not contained in $U_1 \cup U_2$ is equal to a $(U_1 \cap U_2)$-coset.
\end{enumerate}
\end{lem}

\begin{proof}
We let $N=N_{\Aut(\A)}(H_R), \, W=W_1+W_2$ and $U=U_1 \cap U_2$. Note that $\A=V(H,N_0),$ where $N_0$ denotes the stabilizer of
$0$ in $N$. Also notice that, $w^\gamma=w$ for all $w\in W$ and
$\gamma \in N_0$.
Fix a non-trivial element $x \in A_v,$ and some $u_1 \in U_1 \setminus U$ which is not fixed by $x$.
Then there exists an integer $k$ such that $ku_1+v \in U_2$. Since $v \notin U_1,$ it follows that
$ku_1+v \notin U$. We define the elements $v_1=ku_1, \, v_2=ku_1+v$ and $v_3=v_1^x-v_1$.
We find that $v_2^x=v_2+v_3$. For $i\in \{1,2\},$
let $T_i=v_i^A,$ the $A$-orbit containing $v_i$ (in other words, $T_i$ is the basic set of
$\A$ that contains $v_i$). Note that, since $\A_{U_i/W_i}$ is indecomposable, it follows that $\rad(T_i)=W_i$ for both $i=1,2$.
Since $v_3=v_i^x-v_i,$ it follows that $v_3 \in T_i-T_i,$ hence $v_3 \in U_i$.
We conclude that $v_3 \in U_1 \cap U_2$. Suppose that $v_3 \in W$. Then both $v_1$ and $v_1^x$ are in
$T_i \cap \sg{v_3}+v_1,$ and this together with Eq.~\eqref{Eq-T/K} shows that $|T_i \cap \sg{v_3}+v'|=p$ for all $v' \in T_i,$ and hence $\sg{v_3} \le \rad(T_i)$.
It follows that $\sg{v_3} \le \rad(T_1) \cap \rad(T_2)=W_1\cap W_2$.
Thus, $v_3=0,$ that is, $v_1^x=v_1$. This contradicts
that $v_1=ku_1$ and $u_1$ was chosen so that it is not fixed by $x$. We conclude that
$v_3 \in U \setminus W$.

Next, assume for the moment that $v_3^x-v_3 \in W_1$. Let us consider the automorphism
$x^{U_1/W_1}$. First, as $v_1^x-v_1=v_3 \notin W_1,$ we see that $x^{U_1/W_1}$ is not the identity mapping.
On the other hand, $x^{U_1/W_1}$ fixes $W/W_1$ pointwise and the element $W_1+v_3$ (here $W_1+v_3$ is regarded as
an element of the group $U_1/W_1$). By all these we find $|C_{U_1/W_1}(x^{U_1/W_1})|=p^2,$ which implies that
$\A_{U_1/W_1}$ is a nontrivial generalized wreath product, a contradiction to the assumption given in (2).
We conclude that $v_3^x-v_3 \notin W_1$.

Notice that, there is a symmetry between the conditions satisfied by the pairs $(U_1,W_1)$ and $(U_2,W_2)$. Therefore, any statement, which involves the subgroups $U_1,\, U_2,\, W_1$ and $W_2,$ and which can be derived from these conditions, gives always rise to yet another statement that is obtained by replacing $U_1$ with $U_2$ and $W_1$ with $W_2$. In what follows, we will refer to the new statement as {\em the dual counterpart}.
For instance, the statement $v_3^x-v_3 \notin W_1$ has dual counterpart: $v_3^x-v_3 \notin W_2$.
Now, as $v_3^x-v_3 \notin W_1 \cup W_2,$ we can choose elements $v_4 \in W_1$
and  $v_5\in W_2$ such that $v_3^x=v_3+v_4+v_5$.

Now, it follows from the above construction
that $(v_1,v_2,v_3,v_4,v_5)$ is an $\A$-basis.
In this basis, the automorphism $x$ is represented by the matrix as shown in Eq.~\eqref{Eq-mtx}.
\begin{equation}
\label{Eq-mtx}
x=
\begin{pmatrix}
1 & 0 & 1 & 0 & 0 \\
0 & 1 & 1 & 0 & 0 \\
0 & 0 & 1 & 1 & 1 \\
0 & 0 & 0 & 1 & 0 \\
0 & 0 & 0 & 0 & 1 \\
\end{pmatrix}
 \quad
y_1=
\begin{pmatrix}
1 & 0 & 0 & 1 & 0 \\
0 & 1 & 0 & 0 & 0 \\
0 & 0 & 1 & 0 & 0 \\
0 & 0 & 0 & 1 & 0 \\
0 & 0 & 0 & 0 & 1 \\
\end{pmatrix}
\quad
y_2=
\begin{pmatrix}
1 & 0 & 0 & 0 & 0  \\
0 & 1 & 0 & 0 & 1  \\
0 & 0 & 1 & 0 & 0  \\
0 & 0 & 0 & 1 & 0  \\
0 & 0 & 0 & 0 & 1  \\
\end{pmatrix}
\end{equation}

Furthermore, it is straightforward to check that each of $y_1$ and $y_2,$ defined in Eq.~\eqref{Eq-mtx}, acts on $H$ as an automorphism of $\A,$ and therefore, it belongs
to $N_0$. Let $M=\sg{x,y_1,y_2}$. Clearly, $M \le N_0,$ and for $i \in \{1,2\},$ the basic set $T_i$ is equal to the orbit $v_i^M$.

Now, let $z \in N_0 \cap N_{v_1}$. For $i \in \{1,2\},$ let us consider the automorphism $z^{U_i/W_i} \in
\Aut(\A_{U_i/W_i})_0$. The latter group is generated by the element $x^{U_i/W_i},$
and we find $z^{U_i/W_i} \in \big\langle x^{U_i/W_i} \big\rangle$. Moreover, as
$v_1^z=v_1,$ it follows that $z^{U_1/W_1}$ is the identity mapping. All these yield that
$z$ can be written in the following form:
\begin{equation}\label{Eq-mtx2}
z=
\begin{pmatrix}
1 & 0 & 0 & 0 & 0  \\
0 & 1 & a & a(a-1)/2 & b \\
0 & 0 & 1 & a & 0  \\
0 & 0 & 0 & 1 & 0 \\
0 & 0 & 0 & 0 & 1  \\
\end{pmatrix}, \ \ a,b\in GF(p).
\end{equation}
This shows that $|N_0 \cap N_{v_1}| \le p^2,$ and therefore,
$|N|=p^8$ or $p^9$.

Fix an element $v'=kv_1+v_2$ for some $k \in \{1,\ldots,p-1\},$
and let $T$ be the basic set of $\A$ that contains $v'$.
Since $\A=V(H,N_0)$ we can write $T=(v')^{N_0}$.
It follows from Eqs.~\eqref{Eq-mtx} and \eqref{Eq-mtx2} that
$W+v' \subseteq T \subseteq U+v'$. This together with
Lemma~\ref{HM5} yields that $T=W+v'$ or $T=U+v'$. By Eq.~\eqref{Eq-mtx}, $(v')^x=kv_1+kv_3+v_2+v_3=v'+(k+1)v_3$.

Assume at first that $k \ne p-1$. Then $(v')^x \notin v'+W,$ and hence $T=U+v'$. The latter condition together with Theorem~\ref{T-1st-multi} yields that every basic set of $\A$ contained in $\sg{U,v'} \setminus U$ is equal to a $U$-coset.

If $|N|=p^8,$ then
$v_3^{N_0}=v_3+\sg{v_4+v_5},$ and this with the previous paragraph implies that (i) holds with $U_3=\sg{U,-v_1+v_2}$. 
Finally, suppose that $|N|=p^9$. Then for $z$ with $a=1$ and
$b=0$ in Eq.~\eqref{Eq-mtx2}, we find
$(-v_1+v_2)^z=-v_1+v_2+v_3,$ thus $(-v_1+v_2)^{N_0}=
U+(-v_1+v_2),$ and (ii) follows.
\end{proof}

\begin{lem}\label{L-bfly2}
With the notation of Lemma~\ref{L-bfly1}, the S-ring $\A$ is CI.
\end{lem}

\begin{proof}
We keep all notations from the previous proof, and let, in addition,
$W_3=\sg{v_4+v_5}$.
Let $f\in \Iso_0(\A)$ such that $f$ fixes all elements $v_i$ with $i \ne 2,$ and also $-v_2$. Recall that, $(v_1,\ldots,v_5)$ is the $\A$-basis defined in the proof of Lemma~\ref{L-bfly1}.
We settle the lemma by showing that $T^f=T$ for all basic sets $T \in \Bs(\A)$. This will be done in five steps.
\smallskip

\noindent{\bf Claim~(a).} {\it $T^f=T$ for all $T \in \Bs(\A),\, T \subset U_1 \cup U_2$.}
\smallskip

This is trivial for the basic sets $\{v_4\}$ and $\{v_5\},$ and hence
Claim~(a) follows for all basic sets
$T \subset W=\sg{v_4,v_5}$.
Let $N=N_{\Aut(\A)}(H_R)$.
It follows from the proof of Lemma~\ref{L-bfly1} that
any basic set $T \subset U \setminus W$ is in the form
$T=k(W_3+v_3+w)$ for some $w\in W$ and some $k \in \{1,\ldots,p-1\}$ if $|N|=p^8;$ whereas in the form $T=k(W+v_3)$ for some $k \in \{1,\ldots,p-1\}$ if $|N|=p^9$.

As both $W_3+v_3$ and $\{w\}$ are
basic sets, we can write  $(W_3+v_3+w)^f=(W_3+v_3)^f+w^f=W_3^f+v_3^f+w^f=
W_3+v_3+w,$ where the second equality follows by Lemma~\ref{HM3}(ii). Now, this together with Lemma~\ref{L-iso}(i) yields
$T^f=T$ in the case when $|N|=p^8$. Similarly,
$(W+v_3)^f=W+v_3,$ and this with Lemma~\ref{L-iso}(i) yields
$T^f=T$ if $|N|=p^9$.

Next, let us consider the isomorphism $f^{U_1/W_1}$. Since $\A_{U_1/W_1}$ is indecomposable, $f^{U_1/W_1} \in \Aut(U_1/W_1),$ see Lemma~\ref{L-exp}. This together with the fact that $f$ fixes
pointwise a generating set of $U_1$ shows that $f^{U_1/W_1}$ is the identity mapping.
Using this and that $\A_{U_1}$ is a $U_1/W_1$-wreath product, we deduce that $T^f=T$ for all basic sets
$T \subset U_1 \setminus U$. Observe that, the latter statement has dual counterpart: $T^f=T$
for all basic sets $T \subset U_2 \setminus U$. This completes the proof of Claim~(a).
\smallskip

\noindent{\bf Claim~(b).} {\it There exist an integer $k$ and a function $F_2 : U_2 \to \{0,1,\ldots,p-1\}$ such that}
\begin{equation}\label{Eq1-f}
(v_1+u)^f=v_1+u^{x^k}+F_2(u) v_5 \textit{ \ for all \ } u \in U_2.
\end{equation}

The S-ring $\A_{H/U}=\Q \, H/U$. Form this and the fact that $f$ fixes a basis of $H,$ we deduce that
$f^{H/U}$ is the identity mapping, and therefore, $f$ maps any
$U$-coset to itself. In particular, it fixes the coset $U_2+v_1$.
Let $\tilde{f}$ denote the permutation of $U_2+v_1$ induced by the action of $f$ on $U_2+v_1$.
Choose an arbitrary basic set $T \in \Bs(\A_{U_2}),$ and let $\Sigma$ be the
subdigraph of $\Cay(H,T)$ induced by the set
$U_2+v_1$. By Claim~(a), $T^f=T,$ and this in turn implies that
$f \in \Aut(\Cay(H,T)),$ and $\tilde{f} \in \Aut(\Sigma)$.
It is straightforward to check that
$(v_1)_R$ is an isomorphism between $\Cay(U_2,T)$ and $\Sigma$.  Therefore, the permutation $g \in \Sym(U_2),$ defined by $g=(v_1)_R \tilde{f} (-v_1)_R,$ belongs to
$\Aut(\Cay(U_2,T))$. As $T \in \Bs(\A_{U_2})$ was chosen arbitrarily,
by definition, $g \in \Aut(\A_{U_2})$. Furthermore, $0^g=0^{(v_1)_R \tilde{f} (-v_1)_R}=0$. We have
already shown that $\Aut(\A_{U_2/W_2})_0=\big\langle x^{U_2/W_2} \big\rangle,$ and thus we can write
$g^{U_2/W_2}=(x^k)^{U_2/W_2}$ for some integer $k$. This allows us to define the function
$F_2 : U_2 \to \{0,1,\ldots,p-1\},$ by letting $F_2(u)v_5=u^g-u^{x^k}$ for each $u \in U_2$.
Then, 
$u^{x^k}+F_2(u)v_5=u^g=u^{(v_1)_R \tilde{f}(-v_1)_R}=(u+v_1)^f-v_1,$ and Claim~(b) follows.
\smallskip

Recall that $N=N_{\Aut(\A)}(H_R)$.
\smallskip

\noindent{\bf Claim~(c).} {\it Suppose that $|N|=p^8$. Then for any
$u,u' \in U_2$ with $u-u' \in U,$ $F_2(u)=F_2(u')$.}
\smallskip

Suppose that $u,u' \in U_2$ such that $u-u' \in U$.
Recall that $g \in \Aut(\A_{U_2})$ and $g$ fixes any $U$-coset. Let us consider the automorphism
$g^{U_2/W_3}$. It can be easily seen from the proof of 
Lemma~\ref{L-bfly1}(i) that $\A_{U/W_3}=\Q \, U/W_3$. This implies that $g^{U_2/W_3}$ acts on the coset $(U+u)/W_3=(U/W_3)+(W_3+u)$ as a translation by
some element from $U/W_3$. This implies that, $u^g-u+u'-(u')^{\, g} \in W_3$. It follows from Eq.~\eqref{Eq-mtx} that
$u^{x^k}-u+u'-(u')^{\, x^k} \in W_3$ also holds, and by the definition of $F_2$ we can write
$$
(F_2(u)-F_2(u'))v_5=u^g-u^{x^k}-(u')^g+(u')^{x^k} \in W_3.
$$
On the other hand, $\sg{v_5} \cap W_3=\{0\},$ and this yields
Claim~(c).
\medskip

Notice that, the symmetry between the pairs $(U_1,W_1)$ and $(U_2,W_2)$ extends to the triples
$(v_1,T_1,v_4)$ and $(-v_2,-T_2,v_5)$. Here $-T_2=\{-v : v \in T_2\},$ which is also a basic set of
$\A$ (see the definition of an S-ring). As a consequence, the statements in Claims~(b) and (c) have the following dual counterparts:
\smallskip

\noindent{\bf Claim~(d).} {\it Suppose that $|N|=p^8$.
Then there exist an integer $l$ and a function $F_1 : U_1 \to \{0,1,\ldots,p-1\}$ such that}
\begin{equation}\label{Eq2-f}
(-v_2+u)^f=-v_2+u^{x^l}+F_1(u) v_4 \textit{ \ for all \ } u \in U_1.
\end{equation}
{\it Moreover, if $u,u' \in U_1$ with $u-u' \in U,$ then $F_1(u)=F_1(u')$.}
\medskip

We are ready to handle the remaining basic sets of $\A$.
\smallskip

\noindent{\bf Claim~(e).} {\it $T^f=T$ for all $T \in \Bs(\A)$.}
\smallskip

In view of Claim~(a), we can assume that $T \not\subset U_1 \cup U_2$. If $|N|=p^9,$ then $T$ is equal to a $U$-coset, see
Lemma~\ref{L-bfly1}(ii). Using this
and that $f$ maps any $U$-coset to itself, Claim~(e) follows at once
if $N|=p^9$.

In the rest of the proof it will be assumed that $|N|=p^8$.
Let us consider the coefficient of $v_3$ in the linear combination of
$(v_1-v_2)^f$. By Eq.~\eqref{Eq1-f}, this follows to be equal to $-k$. Then, applying Eq.~\eqref{Eq2-f}, we find that this is equal to $l$. We conclude that $l=-k$.

Next, let us consider the element $w=(v_1-v_2+v_3)^f-(v_1-v_2)^f$.
Using Eq.~\eqref{Eq1-f} and Claim~(c), we find
\begin{eqnarray*}
w &=& v_1+(-v_2)^{x^k}+v_3^{x^k}+F_2(-v_2+v_3)v_5-\big(v_1+(-v_2)^{x^k}+F_2(-v_2)v_5\big) \\
&=& v_3^{x^k}=v_3+kv_4+kv_5.
\end{eqnarray*}
On the other hand, using Eq.~\eqref{Eq2-f}, the fact that $l=-k$ and Claim~(d), we find
\begin{eqnarray*}
w &=& -v_2+(v_1)^{x^{-k}}+v_3^{x^{-k}}+F_1(v_1+v_3)v_4-
\big(-v_2+(v_1)^{x^{-k}}+F_1(v_1)v_4\big) \\
&=& v_3^{x^{-k}}=v_3-kv_4-kv_5.
\end{eqnarray*}
We conclude that $k=0$.

Let $U_3=\sg{U,v_1-v_2}$.
We have shown in the proof of Lemma~\ref{L-bfly1}(i) that the
all basic sets of $\A$ not contained in $U_1 \cup U_2 \cup U_3$ are
$U$-cosets, and all basic sets contained in $U_3 \setminus U$ are
$W$-cosets. Combining this with the fact that $f$ maps any
$U$-coset to itself, we get that Claim~(e) holds whenever $T \not\subset U_1 \cup U_2 \cup U_3$. It remains to check whether the basic sets contained in $U_3 \setminus U$ are fixed by $f$.
By Theorem~\ref{T-1st-multi}, any such a basic set can be written in the form
$T^{(k)}=kT=\{ku : u \in T\},$ where $k \in \{1,\ldots,p-1\}$ and $T\subset U_3 \cap (U_2+v_1)$.
Then, By Lemma~\ref{L-iso}(i),
we are done if we show that $T^f=T$. Since $T\subset U_2+v_1,$ there exists some $u\in U_2$ such that
$T=W+v_1+u$. Now, applying Eq.~\eqref{Eq1-f} and using that $k=0,$ we finally get
$$
T^f=W+(v_1+u)^f=W+v_1+u+F_2(u)v_5=W+v_1+u=T.
$$
This completes the proof of Claim~(e), and thus the proof of the lemma as well.
\end{proof}

\begin{lem}\label{L-AU-ind}
Assuming Hypothesis~\ref{hyp}, suppose that $\A$ is an $U/W$-wreath product, where $|U|=p^4, \, |W|=p$ and $\A_U$ is indecomposable.
Then the S-ring $\A$ is CI.
\end{lem}

\begin{proof}
If $|A|=p,$ then each basic set has size at most
$p,$ and therefore, each basic set $T \subset H \setminus U$ is equal to a $W$-coset. This implies that $\A_{H/W} = \Q \, H/W$.
In particular, $\A_{H/W}$ is indecomposable, and we can apply Corollary~\ref{C-teq-min} to get that $\A$ is a CI-S-ring. For the rest of the proof we assume that $|A| \ge p^2$.

Let $A^U$ denote the group of automorphisms of $U$ induced by
restricting $A$ to $U$. We show next that $|A^U| \le p$.
Assume to the contrary that $|A^U| > p$. Since $\A_U$ is indecomposable, it follows by
Lemma~\ref{L-center1} that $|\O(\A)|=p^2$. Fix $u_1,u_2  \in \O(\A)$ such that
$\O(\A)=\sg{u_1,u_2}$. Now, let $V < U$ be an $\A$-subgroup such that
$|V|=p^3$ and $\O(\A) < V,$ and fix an element $u_3 \in V \setminus \O(\A)$.
If $A^U$ is not semiregular on the orbit $u_3^A,$ then it follows that
$|C_U(z)|=p^3$ for any nontrivial $z\in (A^U)_{u_3}$. This contradicts Lemma~\ref{L-center1}, and thus $u_3^{\, A}=\O(\A)+u_3$ and
$|A^U|=p^2$. There exist unique elements $x,y \in A$ that satisfy
$u_3^x=u_1+u_3$ and $u_3^y=u_2+u_3$. Now, let $u$ be an arbitrary element from
$U \setminus V$. Then both $u^x-u$ and $u^y-u$ are in $V$. Therefore,
there are integer numbers $k,l,m,k',l',m' \in \{0,1,\ldots,p-1\}$ for which
$$
u^x=u+ku_1+lu_2+mu_3 \text{ and } u^y=u+k'u_1+l'u_2+m'u_3.
$$
Since $|A^U|=p^2,$ the group $A^U$ is abelian, and
$u^{xy}=u^{yx}$. The coefficients of $u_1$ and $u_2,$ resp.,
have to be the same in both sides, and this results in the equalities:
$k+k'=k+k'+m'$ and $l+l'+m=l+l',$ resp., which gives that
$m=m'=0$. This implies that the orbit $u^{\sg{x,y}}=\O(\A)+u,$ and
therefore, $\rad(T) \ge \O(\A)$ for the basic set of $\A$ that contains $u$. As $u$ was chosen arbitrarily from $U \setminus V,$ we obtain finally that $\A_U$ is a
$V/\O(\A)$-wreath product. This contradicts the assumption that $\A_U$ is indecomposable, and by this we have proved that $|A^U| \le p$.

Let us fix $T_1 \in \Bs(\A)$ such that $T_1 \subset H \setminus U$ and
$|\rad(T_1)|$ is the the smallest among all
$|\rad(T)|$ where $T$ runs over the set of all basic sets
$T\in \Bs(\A), \, T \subset H \setminus U$.
Now, let $(v_1\ldots,v_5)$ be an $\A$-basis such that
$v_i \in U$ for all $i < 5$ and $v_5 \in T_1$.
Let $f \in \Iso_0(\A)$ such that $f$ fixes $v_i$ for all $i \in \{1,\ldots,5\}$.
We settle the lemma by showing that $T^f=T$ for all basic sets
$T \in \Bs(\A)$.

Since $f$ fixes all $v_i$ and $U=\sg{v_1,\ldots,v_4},$ the $\A$-subgroup $U$ is mapped to itself by $f$.
Then $f^U$ is a normalized isomorphism of $\A_U$. Since $\A_U$ is indecomposable, by Corollary~\ref{C-normal}, $f^U \in \Aut(U),$ which implies that $f^U$ is the identity mapping. In particular,
$T^f=T$ for all basic sets $T \subset U$.

Now, we turn to the basic sets contained in $H \setminus U$.
Let
$$
V=\bigcap_{T\in \Bs(\A), \, T \subset H\setminus U}\rad(T).
$$
Let us consider the S-ring $\A_{H/V}$.
Since $V \le U$ and $f$ fixes $U$ pointwise, we get $V^f=V,$
and thus $f^{H/V} \in \Iso_0(\A_{H/V})$ (see Proposition~\ref{HM3}(iii)).
We finish the proof by showing that
$f^{H/V}=\id_{H/V}$. Indeed, then any basic set $T \subset H \setminus U$ satisfies  $(T+V)^f=T+V,$ and since $V \le \rad(T),$ it follows by Lemma~\ref{L-iso}(ii) that $T^f=T$.
\medskip

\noindent{\bf Case~1.} $|A|=p^2$.
\smallskip

Let $V_1=\rad(T_1)$. Clearly, $V \le V_1$.
We show at first that we may choose $v_5$ such that
\begin{equation}\label{Eq-T1}
T_1=V_1+v_5 \text{ and } V=V_1.
\end{equation}

Since $|A^U|\le p,$ there exists $x \in A$ such that $C_H(x)=U$. Then the $\sg{x}$-orbits not contained in $U$
coincide with the cosets of a subgroup $W' < H$ of order $p,$ in particular, $\A$ is a
$U/W'$-wreath product. We may assume  w.~l.~o.~g. that $W'=W$. Let $A=\sg{x,y}$.
Let us consider the S-ring $\A_{H/W}$. It follows that
$\A_{H/W}=V(H/W,\sg{y^{H/W}}),$ where $y^{H/W}$ is
the automorphism of $H/W$ induced by the action of
$y$ on $H/W$. In view of Corollary~\ref{C-teq-min} we may assume
that $\A_{H/W}$ is decomposable. This implies that any
$\sg{y^{H/W}}$-orbit is contained in a coset of a fixed subgroup of
$H/W$ of order $p$. It is not hard to show that this implies
that $|C_{H/W}(y^{H/W})|=p^3$. Let $U'$ be the unique subgroup of $H$ that contains $W$ and for which
$U'/W=C_{H/W}(y^{H/W})$.

Suppose at first that $U'=U$. Then, any $\sg{y}$-orbit in $U$ is contained
in a $W$-coset, and hence $|C_U(y)| \ge p^3$.  On the other hand,
$\A_U$ indecomposable, and it follows
from Lemma~\ref{L-center1} that $y$ acts as the identity on $U$. Hence, $C_H(y)=U,$ and so $C_H(A)=U$. This shows that any basic set $T \subset H\setminus U$ is equal to a $V$-coset, in particular, Eq.~\eqref{Eq-T1} follows.

Next, suppose that $U' \ne U,$ and choose an element
$v \in U' \setminus U$. Then the basic set $T(v)$ containing $v$
is equal to the coset $W+v$. Thus, $|\rad(T(v))|=p,$ which is
clearly minimal among all orders $|\rad(T)|, \, T \subset H\setminus U$. Then choosing
$v_5$ to be $v,$ we get $T_1=T(v)=W+v_5$ and also $W=V_1=V,$ that is, Eq.~\eqref{Eq-T1} holds also in this case.

Let $U''=\sg{V,v_5}$.
The group $H/V$ decomposes to the internal direct sum $H/V=\bar{U}+\bar{U}'',$ where both factors $\bar{U}=U/V$ and $\bar{U}''=U''/V$ are $\A_{H/V}$-subgroups.
To simplify notation, we write $\bar{v}$ for the coset
$V+v,$ where $v$ is any element in $H,$ 
and $\bar{S}$ for the set $S/V \subseteq H/V,$ where
$S \subseteq H$. Notice that, $f^{H/V}$ fixes pointwise both $\bar{U}$ and
$\bar{U}''$. Let $\bar{v} \in H/V$ be an arbitrary element. Then, $\bar{U}+\bar{v}=\bar{U}+\bar{v}_1$ for some $\bar{v}_1 \in \bar{U}'',$ and we can write
$$
\bar{U}+\bar{v}^{f^{H/V}}=(\bar{U}+\bar{v})^{f^{H/V}}=
(\bar{U}+\bar{v}_1)^{f^{H/V}}=\bar{U}+\bar{v}_1=
\bar{U}+\bar{v}.
$$
Similarly, $\bar{U}''+\bar{v}=\bar{U}''+\bar{v}_2$ for
some $\bar{v}_2 \in \bar{U},$ and hence
$$
\bar{U}''+\bar{v}^{f^{H/V}}=(\bar{U}''+\bar{v})^{f^{H/V}}=
(\bar{U}''+\bar{v}_2)^{f^{H/V}}=\bar{U}''+\bar{v}_2=
\bar{U}''+\bar{v}.
$$
Therefore,
$\bar{v}^{f^{H/V}}-\bar{v} \in \bar{U} \cap \bar{U}'' = \{\bar{0}\}
,$ and $f^{H/V}=\id_{H/V},$ as claimed.
\medskip

\noindent{\bf Case~2.} $|A| \ge p^3$.
\smallskip

Let $K$ denote the kernel of $A$ acting on the set $U$.
Since $|A^U| \le p,$ $|K|\ge p^2$. Since $K \le \Aut(H)$ and each
element of $K$ fixes $U$ pointwise, there exists a subgroup $V' \le U$ such that $|V'|=|K| \ge p^2,$ and the $K$-orbits not contained in $U$ are equal to the $V'$-cosets not contained in $U$.
Note that, $V \ge V',$ and thus $|V_1|\ge |V| \ge p^2,$ where $V_1=\rad(T_1)$.

Suppose at first that $T_1 \ne V_1+v_5$.
By  Proposition~\ref{HM2}, $|T_1|=p^3,$  which implies that
$|V_1|=p^2,$ and hence $V=V_1$. Also, the S-ring $\A_{H/V}$ is an exceptional S-ring over the group $H/V,$ and it follows that $f^{H/V} \in \Aut(H/V)$. Since $f$ fixes all generators $v_i,$
it follows that $f^{H/V}=\id_{H/V}$.

Now, suppose that $T_1=V_1+v_5$. It is sufficient to
show that $V=V_1$. Then the conditions in Eq.~\eqref{Eq-T1} hold, and $f^{H/V}=\id_{H/V}$ follows as in Case~1.
Since  $V \le V_1, \, |V| \ge p^2$ and $|V_1| \le p^3,$ the equality $V=V_1$ follows if $|V_1|=p^2$. We are left with the case  when $|V_1|=p^3$. Then, for each basis set $T \subset H\setminus U,$ $|\rad(T)|=p^3,$
implying that $T$ is a coset of a subgroup that contains $V$.
Assume to the contrary that $V\ne V_1$. Then $|H/V|=p^3,$ and $\A_{H/V}$ has two basic sets which are cosets of distinct subgroups of order $p$. Since $|H/V|=p^3,$ the S-ring $\A_{H/V}$ is Cayley isomorphic to one of the S-rings given in Table~1. A quick look at the table shows that none of these S-rings has two basic sets which are cosets
of distinct subgroups of order $p$. Therefore, $V=V_1,$ and this completes the proof of the lemma.
\end{proof}

Before we prove Theorem~\ref{31}, one more technical lemma is needed.

\begin{lem}\label{L-gwp}
Assuming Hypothesis~\ref{hyp}, suppose that $\A$ is a $U/W$-wreath product, where $|U|=p^4$ and $|W|=p$.
Furthermore, let $V$ be an
$\A$-subgroup such that $W < V < U, \, |V|=p^3$ and $|\O(\A) \cap V| \ge p^2,$ and let $f \in \Iso_0(\A)$.
Then there exists $\varphi \in \Aut(H)$ such that $f \varphi$ maps each of $U$ and $W$ to itself, and the following conditions hold:
\begin{enumerate}[(i)]
\item $T^{f \varphi}/W=T/W$  for all  $T \in \Bs(\A)$.
\item $T^{f \varphi}=T$ for all  $T \in \Bs(\A)$ with
$T \subset V \cup (H\setminus U)$.
\item If either $V \subseteq \O(\A),$ or each
basic set of $\A$ contained in $V \setminus \O(\A)$ is equal to
a $W$-coset, then $f \varphi$ fixes pointwise $V$.
\end{enumerate}
\end{lem}

\begin{proof}
Let us fix five elements of $H$ as follows:
$$
v_5 \in W, \, v_4 \in (V \cap \O(\A)) \setminus W,\, v_3 \in V \setminus \sg{v_4,v_5}, \, v_2 \in U \setminus  V
\text{ and } v_1 \in H \setminus U.
$$
Clearly, $(v_1,\ldots,v_5)$ is an $\A$-basis.
Let $\varphi_1 \in \Aut(H)$ be the automorphism defined
by $\varphi : v_i^f \mapsto v_i, \, i \in \{1,\ldots,5\},$ and
let $f_1=f \varphi$. It follows that $f_1 \in \Iso_0(\A),$ and
$f_1$ fixes all $v_i$. In particular, $W^{f_1}=W,$ and thus
$f_1^{H/W} \in \Iso_0(\A_{H/W})$.

The S-ring $\A_{H/W}$ is a CI-S-ring. Therefore, there exists some $\phi \in \Aut(H/W)$ such that $T^{f_1}/W=(T/W)^\phi$ for all $T\in \Bs(\A)$.  Let $\phi_1 \in \Aut(H)$ such that $\phi_1$ fixes $W$ pointwise, and
$\phi_1^{H/W}=\phi$. Then we have
$W+v_4^{\phi_1}=(W+v_4)^{\phi_1^{H/W}}=(W+v_4)^\phi=W+v_4^{f_1}=W+v_4$.
Thus, $v_4^{\phi_1}=v_4+w$ for some $w \in W$. Now, let $\phi_2 \in \Aut(H)$ be defined by
$$
\phi_2 : v_i \mapsto \begin{cases}
v_i & \text{if } i=1,2,3,5, \\
v_i-w & \text{if } i=4, \end{cases}
$$
and let $\varphi_2=\phi_1\phi_2$.
It follows that $\varphi_2$ satisfies the following:
\begin{enumerate}[(a)]
\item $u^{\varphi_2}=u$ for all $u \in \sg{v_4,v_5},$
\item$\varphi_2^{H/W}=\phi$.
\end{enumerate}

Let $T_1 \in \Bs(\A)$ such that $v_3 \in T_1$.
Then $T_1=X+v_3$ for some subgroup $X \le \sg{v_4,v_5},$ and
we find $T_1^{f_1}=X^{f_1}+v_3^{f_1}=
X+v_3=T_1$. Thus, $T_1^{\varphi_2}/W=(T_1/W)^{\varphi_2^{H/W}}
=(T_1/W)^{\phi}=T_1^{f_1}/W=T_1/W$. This gives $T_1^{\varphi_2}+W=T_1+W,$ and thus
$T_1^{\varphi_2^{-1}}$ is contained in the coset $T_1+W=X+W+v_3$.
By (a), $X^{\varphi_2}=X,$ implying that
$T_1^{\varphi_2^{-1}}=X+v_3^{\varphi_2^{-1}}$. This with the previous observation yields that $T_1^{\varphi_2^{-1}}=T_1+w_1$ for some $w_1 \in W$. Now, we define $\varphi_3 \in \Aut(H)$ by letting
$$
\varphi_3 : v_i \mapsto \begin{cases}
v_i & \text{if } i=1,2,4,5, \\
v_3+w_1 & \text{if } i=3 \text{ and } X \ne W, \\
v_3^{\varphi_2^{-1}} & \text{if } i=3 \text{ and } X=W.
\end{cases}
$$

Let $\varphi=\varphi_1\varphi_2^{-1}\varphi_3^{-1}$.
Then $f \varphi=f_1 \varphi_2^{-1} \varphi_3^{-1}$.
It is easily seen that $f_1$ maps each of
$U$ and $W$ to itself. The condition that $\varphi_2$ maps $W$ to itself
follows from (a) and the fact that $W=\sg{v_5}$. Then, $W \le U^{\varphi_2},$ and for
$U^{\varphi_2}=U$ it is enough to show that
$U^{\varphi_2}/W=U/W$. This follows along the line:
$U^{\varphi_2}/W=(U/W)^{\varphi_2^{H/W}}=(U/W)^\phi=U^{f_1}/W=U/W$.
The definition of $\varphi_3$ shows that it also maps $U$ and $W$ to itself, and therefore,
we obtain that $f \varphi$ maps each of $U$ and $W$ to itself.
We finish the proof by showing that all conditions (i)-(iii) hold for $f \varphi$. \medskip

(i): Recall that $f_1=f \varphi_1$.
It follows from (b) that $T^{f_1\varphi_2^{-1}}/W=T/W$ for all
$T \in \Bs(\A)$. We claim that
$(W+v_i)^{\varphi_3}=W+v_i$ for al $i\in \{1,\ldots,5\}$.
This is obvious if $i \ne 3,$ or $i=3$
and $X \ne W$.  Let $i=3$ and $X=W$. Then
$T_1=W+v_3,$ and since $T_1^{f_1}=T_1,$ it follows that
$T_1^{\varphi_2}/W=T_1^{f_1}/W=T_1/W,$ implying that
$W+v_3^{\varphi_2^{-1}}=W+v_3$.
Therefore, $(W+v_3)^{\varphi_3}=W^{\varphi_3}+v_3^{\varphi_3}=W+v_3^{\varphi_2^{-1}}=W+v_3$.
Since $H=\sg{v_1,\ldots,v_5},$  it follows that
$(W+x)^{\varphi_3}=W+x$ for all $x \in H$. Consequently,
$T^{f \varphi}/W=T^{f_1\varphi_2^{-1} \varphi_3^{-1}}/W=T^{\varphi_3^{-1}}/W=T/W$ for all
$T \in \Bs(\A),$ and (i) follows.
\medskip

(ii): Let $T$ be an arbitrary basic set of $\A$. Note that, if
$W \le \rad(T),$ then using that $f \varphi$ maps $W$ to itself,
we can write $W+T^{f \varphi}=(W+T)^{f \varphi}=T^{f \varphi}$.
Combining this with (i) yields $T^{f \varphi}=T^{f \varphi}+W=T+W=T$.
In particular, (ii) holds whenever $T \subset H \setminus U$ or
$T=T_1$ and $W=X$.
Also, by (a) and the definition of $\varphi_3$, $f \varphi$ fixes pointwise $\sg{v_4,v_5},$ and this gives that (ii) also
holds when $T \subset \sg{v_4,v_4}$. It remains to consider the
case when $T \subset V \setminus \sg{v_4,v_5}$. Observe that, $T$ can be written in the form $T=kT_1+v$ for some
$k \in \{1,\ldots,p-1\}$ and some $v \in \sg{v_4,v_5}$. In view of
Eq.~\eqref{Eq-eT} and Theorem~\ref{T-1st-multi},
in order to prove $T^{f \varphi}=T$ it is sufficient to show that
$T_1^{f \varphi}=T_1$. We have already shown above that this holds if
$X=W$. If $X \ne W,$ then the statement follows along the following line:
$$
T_1^{f \varphi}=T_1^{f_1\varphi_2^{-1}\varphi_3^{-1}}=
 T_1^{\varphi_2^{-1}\varphi_3^{-1}}=(T_1+w_1)^{\varphi_3^{-1}}=
(X+v_3+w_1)^{\varphi_3^{-1}}=X+v_3=T_1.
$$
\medskip

(iii): Since $V=\sg{v_3,v_4,v_5},$ we are done if we show that
$v_i^{f \varphi}=v_i$ for all $i \in \{3,4,5\}$. Using that $f \varphi=f_1 \varphi_2^{-1}
\varphi_3^{-1}$ and all $v_i$ are fixed by $f_1,$ it reduces to show
that $v_i^{\varphi_2^{-1}\varphi_3^{-1}}=v_i$ for all $i \in \{3,4,5\}$.
This follows immediately from (a) and the definition of $\varphi_3$ if $i=4$ or $i=5$.
The condition $V \subset \O(\A)$ is equivalent to $T_1=\{v_3\},$ and
therefore, $T_1=\{v_3\}$ or $T_1=W+v_3$.
Now, recall that $T_1^{\varphi_2^{-1}}=T_1+w_1$.
This shows that, if $T_1=\{v_3\},$ then $v_3^{\varphi_2^{-1}}=v_3+w_1,$ and so $v_3^{\varphi_2^{-1}\varphi_3^{-1}}=(v_3+w_1)^{\varphi_3^{-1}}=v_3$. Finally, if $T_1=W+v_3,$ then
$v_3^{\varphi_3}=v_3^{\varphi_2^{-1}},$ that is,
$v_3^{\varphi_2^{-1}\varphi_3^{-1}}=v_3$.
\end{proof}

Everything is prepared to settle the main result of the section.

\begin{proof}[Proof of Theorem~\ref{31}]
Since $\A$ is decomposable, it is a $U/W$-wreath product where
$W < U, \, |W|=p$ and $|U|=p^4$.  Furthermore,
we may assume because of Lemma~\ref{L-AU-ind} that
$\A_U$ is a $V/X$-wreath product where $X < V < U, \, |X|=p$
and $|V|=p^3$.

Let $f \in \Iso_0(\A)$. We have to show that
there exists some $\varphi \in \Aut(H)$ such that
$T^{f \varphi}=T$ for all basic sets $T\in \Bs(\A)$.
By Lemma~\ref{L-gwp}(i)-(ii), there exists some
$\varphi_1 \in \Aut(H)$ such that $f_1=f \varphi_1$
satisfies
\begin{equation}\label{Eq-cond1}
T^{f_1}/W=T/W  \text{ for all } T \in \Bs(\A),
\end{equation}
and
\begin{equation}\label{Eq-cond2}
T^{f_1}=T \text{ for all  } T \in \Bs(\A) \text{ with } T \subset V \cup H\setminus U.
\end{equation}

Now, if $W \le \rad(T)$ for any basic set $T \subset V \setminus U,$
then Eq.~\eqref{Eq-cond1} and Lemma~\ref{L-iso}(ii) imply that
$T^{f_1}=T$ also holds, and hence we are done by letting
$\varphi=\varphi_1$.
For the rest of the proof it will be assumed that there exists
some basic set $T_1 \subset U \setminus V$ such that $W \not\le \rad(T_1)$. Note that, this implies that
$|T_1|=p$ or $|T_1|=p^2$.
We consider below the two cases separately.
\smallskip

\noindent {\bf Case~1.} $|T_1|=p$.
\smallskip

In this case $T_1=X+u$ for some $u \in U \setminus V$.
Since $W \not\le \rad(T_1),$ it follows that $X \ne W$. Since $X < V,$ by Eq.~\eqref{Eq-cond2}, $X^{f_1}=X,$ and thus $T_1^{f_1}=X^{f_1}+u^{f_1}=X+u^{f_1}$. Using this and Eq.~\eqref{Eq-cond1}, we conclude that
 $T_1^{f_1}+W=T_1+W,$ implying that $T_1^{f_1}$ is contained in the coset $X+W+u$. Thus, $T_1^{f_1}=T_1+w_1$ for some
$w_1 \in W$. Choose an automorphism $\varphi_2 \in \Aut(H)$ satisfying the following:
$$
\varphi_2 :u \mapsto u-w_1,\;
|C_H(\varphi_2)|=p^4 \text{ and }
C_H(\varphi_2) \cap U=V.
$$

Let $f_2=f_1 \varphi_2$.  It is not hard to show that $f_2$ satisfies both Eqs.~\eqref{Eq-cond1} and \eqref{Eq-cond2}.
We show below that $T^{f_2}=T$ also holds for each $T\in \Bs(\A)$ with $T\subseteq U \setminus V,$
and therefore, will be done by letting $\varphi=\varphi_1 \varphi_2$.

First, $T_1^{f_2}=(T_1+w_1)^{\varphi_2}=(X+u+w_1)^{\varphi_2}=T_1$. Let us consider the S-ring $\A_{U/X}$ and its isomorphism $f_2^{U/X}$. Note that, the latter isomorphism belongs to 
$\Iso_0(\A_{U/X})$ because $f_2$ maps each of $U$ and $X$ to itself.
The group $U/X$ can be written as the internal direct sum
$V/X+\sg{X,u}/X$ where both subgroups $V/X$ and $\sg{X,u}/X$ are
$\A_{U/X}$-subgroups. By Lemma~\ref{L-tensor},
$\A_{U/X}=\A_{V/X} \otimes \A_{\sg{X,u}/X}$.
Also, $f_2^{U/X}$ fixes any basic
set $T' \in \Bs(\A_{U/X})$ if $T' \subset V/X$ or $T'\subset \sg{X,u}/X$. This together with $\A_{U/X}=\A_{V/X} \otimes \A_{\sg{X,u}/X}$ yields that $f_2^{U/X}$ fixes all basic sets of $\A_{U/X}$.
Then, using also that $\A_U$ is a $V/X$-wreath product, we conclude $f_2$ fixes all basic sets of $\A$ contained in
$U \setminus V,$ as required.
\smallskip

\noindent {\bf Case~2.} $|T_1|=p^2$.
\smallskip

Assume for the moment that $T_1$ is equal to a coset $X'+u$
for some subgroup $X' < V, \, |X'|=p^2,$ and $|W \cap X'|=1$.
It follows from Eq.~\eqref{Eq-cond1} that $T_1^{f_1}=T_1+w_1$ for some $w_1 \in W$. Repeating the argument used in Case~1, we find $\varphi_2\in \Aut(H)$ such that $f_2=f_1 \varphi_2$ satisfies
both Eqs.~\eqref{Eq-cond1} and \eqref{Eq-cond2}, and also $T_1^{f_2}=T_1$. Now, if $T \in \Bs(\A)$ is an arbitrary basic set such that $T \subset U \setminus V,$ then it can be written in the form $T=k(T_1+w_2)$ for some $k \in \{1,\ldots,p-1\}$ and some
$w_2 \in W$. In view of Lemma~\ref{L-iso}(i), $T^{f_2}=T$ follows
because $(T_1+w_2)^{f_2}=T_1^{f_2}+w_2^{f_2}=T_1+w_2$.
All these show that we are done by choosing $\varphi=\varphi_1 \varphi_2$.

For the rest of the proof it will be assumed that $T_1$ is not a coset.
Equivalently, $\A_{U/X}$ is an exceptional S-ring.
This implies that any basic set in $U \setminus V$ generates
$U,$ and therefore, $V$ is the only $\A$-subgroup
of order $p^3$ contained in $U$. This fact will be used later.

By Corollary~\ref{C-teq-min}, the S-ring $\A_{H/W}$ must be decomposable. Equivalently, there exist
$\A$-subgroups $Y$ and $U'$ such that
$W < Y <  U', \, |Y|=p^2$ and $|U'|=p^4,$ and
$\A_{H/W}$ is a $(U'/W)/(Y/W)$-wreath product.
\smallskip

\noindent {\bf Case~2.1.} $U'=U$.
\smallskip

Assume that $Y=W+X$. In this case $\A$ is a $U/X$-wreath
product, and we can finish this case by replacing first
$W$ with $X,$ and then apply the argument used right after
Eq.~\eqref{Eq-cond2}.

Now, suppose that $Y \ne W+X$. Since $W < Y,$ this gives $X \cap Y=\{0\}$ and $|X+Y|=p^3$. 
Then $Y < U'=U$. We obtain that $X+Y$ is an $\A$-subgroup of
order $p^3$ contained in $U$. As it was noted above, this forces that $X+Y=V,$ in particular,
$Y < V$. It follows that, either $V \subseteq \O(\A),$ or
any basic set contained in $V \setminus (W+X)$ is equal to
a $W$-coset. By Lemma~\ref{L-gwp}(iii), $f_1$ fixes pointwise $V$.
Let us consider $f_1^{U/X},$ the isomorphism of $\A_{U/X}$
induced by $f_1$ acting on $U/X$. Then,
$f_1^{U/X} \in \Iso_0(\A_{U/X}),$ and since $\A_{U/X}$ is exceptional, $f_1^{U/X} \in \Aut(U/X),$ see Lemma~\ref{L-exp}.
Also, as $f_1$ fixes pointwise $V,$
$f_1^{U/X}$ centralizes $V/X$. This implies that
$f_1^{U/X}$ acts on $U/X \setminus V/X$ as a translation by some element $X+u_1$ where $u_1\in V$. In particular, $T_1^{f_1}/X=(T_1+u_1)/X,$ and
thus $T_1^{f_1}+X=T_1+u_1+X$. Since $X \le \rad(T_1),$ it
follows that $X=X^{f_1} \le \rad(T_1^{f_1}),$  and we can
write $T_1^{f_1}=T_1^{f_1}+X=T_1+u_1+X=T_1+u_1$.
Using also that $|X \cap Y|=1,$ we can further write
$T_1+u_1=T_1+u_2$  for some $u_2 \in Y$.
Fix an element $u \in U \setminus V,$ and then choose an
automorphism $\varphi_2 \in \Aut(H)$ such that
$$
u^{\varphi_2}=u-u_2,\;
|C_H(\varphi_2)|=p^4 \text{ and }
C_H(\varphi_2) \cap U=V.
$$

Let $f_2=f_1\varphi_2$. Using that $u_2\in Y$ and $Y \le \rad(T)$
for all $T \in \Bs(\A)$ with $T \subset H\setminus U,$
it is not hard to show that $f_2$ satisfies Eq.~\eqref{Eq-cond2}.  Also, $T_1^{f_2}=(T_1+u_2)^{\varphi_2}=T_1$.
Finally, let $T \in \Bs(\A)$ be an arbitrary basic set such that
$T \subset U \setminus V$. Then, $T$ can be written in the form
$T=T_1+w_2, \, w_2 \in W,$ and thus we can write $T^{f_2}=
T_1^{f_2}+w_2^{f_2}=T_1+w_2=T$.
All these show that we are done by choosing $\varphi=\varphi_1 \varphi_2$.
\smallskip

\noindent {\bf Case~2.2.} $U' \ne U$.
\smallskip

Recall that, $V$ is the only
$\A$-subgroup of order $p^3$ contained in $U$. This implies
that $U \cap U'=V$. The radical $\rad(T_1/W)=(X+W)/W$.
On the other hand, $Y/W \le \rad(T_1/W),$ and thus $Y=X+W$.
This shows that $\A$ is a $U'/X$-wreath product.
The S-ring $\A_{U'}$ is a $U'/W$-wreath product, and we
may assume that $\A_{U'/W}$ is indecomposable. Observe that,
letting $U_1=U, \, U_2=U', \, W_1=X$ and $W_2=W,$ 
conditions (1) and (2) of Lemma~\ref{L-bfly1} hold. Therefore, in view of Lemma~\ref{L-bfly2},
we may assume that condition (3) does not hold, that is,
$A$ acts regularly on any of its orbits not contained in
$U \cup U'$. Now, fix an $\A$-basis $(v_1,\ldots,v_5)$ as follows:
$$
v_1\in U \setminus V, \; v_2 \in U'\setminus V,  \; v_3 \in
V \setminus (W+X), \; v_4 \in X \text{ and } v_5 \in W.
$$

Then, let $\psi$ be the automorphism of $H$ defined by
$\psi : v_i^f \mapsto v_i$ for all $i\in \{1,\ldots,5\},$ and let
$f_3=f \psi$. We finish the proof by showing that $T^{f_3}=T$ for all basic sets $T \in \Bs(\A)$. One can settle the equality $T^{f_3}=T$ for $T \subset U \cup U'$ by copying the argument used in the proof of 
Claim~(a) in the proof of Lemma~\ref{L-bfly2}. 

Now, suppose that $T \in \Bs(\A)$ with $T \subset H \setminus
(U \cup U')$.  By Lemma~\ref{L-p-S2}, $T$ is contained in both a
$U$-coset and a $U'$-coset, hence it is contained in a $V$-coset, recall that $V=U \cap U'$.  We claim that $T$ is, in fact, equal to a $V$-coset.
Assume to the contrary that $T$ is properly contained in a $V$-coset.
In particular, we have $|T| \le p^2$.
Let $G=H_R A,$ $N=N_{\Aut(\A)}(H_R),$ and $N_0$ be the stabilizer of $0$ in $N$. Then, $\Aut(\A)=G^{(2)}$ and $C_{H_R}(A) \le Z(G)$. By Proposition~\ref{P-center-G2}, $C_{H_R}(A) \le Z(\Aut(\A)),$ and
hence $C_{H_R}(A) \le C_{H_R}(N_0)$. The S-ring
$\A$ can be expressed as $\A=V(H,N_0)$ such that
$p^2 \le |C_H(N_0)|$. Now, we can replace $A$ with $N_0,$ and
assume, in addition, that $N_0$ is regular on $T$. Therefore,
$|N_0|=|T| \le p^2$. On the other hand, $N_0$ contains two
elements $x$ and $x'$ such that $C_H(x)=U$ and
$x$ acts on the elements $v \in H \setminus U$ as
the right translation $v^x=v+w$ for a fixed nonzero
$w \in W,$ and $C_H(x')=U'$ and
$x'$ acts on the elements $v \in H \setminus U'$ as
the right translation $v^x=v+w'$ for a fixed nonzero
$w' \in X$. It is easily seen that
$V(H,\sg{x,x'})  \ne \A=V(H,N_0),$ implying that
$\sg{x,x'} < N_0$. This, however, contradicts the previously
obtained bound $|N_0| \le p^2,$ and by this we have proved that
$T$ is indeed equal to a $V$-coset.

Finally, let us consider the S-ring $\A_{H/V},$ and the induced
isomorphism $f_3^{H/V} \in \Iso_0(\A_{H/V})$.  It is easily seen
that $\A_{H/V}=\Q \, (H/V),$
and therefore, $f_3^{H/V} \in \Aut(H/V)$. Using this and that
$f_3$ fixes all $v_i, \, i \in \{1,\ldots,5\},$ we find
$f_3^{H/V}=\id_{H/V}$. Thus, $(T+V)^{f_3}=T+V,$ and combining
this with $V \le \rad(T),$ Lemma~\ref{L-iso}(ii) gives that $T^{f_3}=
T$. This completes the proof of Case~2.2.
\end{proof}

\section{Proof of Theorem~\ref{3} II: The indecomposable S-rings}

In this section, we turn to the indecomposable S-rings in 
Theorem~\ref{3} and prove the following theorem:

\begin{thm}\label{32}
Assuming Hypothesis~\ref{hyp}, suppose that $\A$ is indecomposable.
Then $\A$ is a CI-S-ring.
\end{thm}

\noindent
Our main result Theorem~\ref{3} follows then as the consequence of Theorems~\ref{31} and \ref{32}.
\medskip

Theorem~\ref{32} will be proved in the end of the section after six preparatory lemmas.
In the first two lemmas we derive some properties of indecomposable $p$-S-rings
with thin radical of order at least $p^2$.

\begin{lem}\label{L-center2}
Let $\B$ be an indecomposable $p$-S-ring
over a group $H \cong \Z_p^5$ such that $|\O(\B)|\ge p^2,$
and let $x \in N_{\Aut(\B)}(H_R)$ such that $x\ne \id_H$ and
$0^x=0$. Then $|C_H(x)| \le p^3$.
\end{lem}

\begin{proof}
Since $x\ne \id_H,$ it follows that $C_H(x) \le p^4$.
Assume to the contrary that $|C_H(x)|=p^4$.
Let $U=C_H(x),$ and for a fixed $v_1 \in H \setminus U,$ let
$W = \sg{v_1^x-v_1}$. Then $|W|=p,$ and it
follows that the orbit $v^{\sg{x}} = W+v$ for all $v \in H \setminus U$.
It can be shown in the same way as in the proof of
Lemma~\ref{L-center1} that neither $U$ nor $W$ are $\B$-subgroups.

Let $V_1 \le \O(\B)$ such that $|V_1|=p^2$ and let
$V_2=V_1+W$. Then $V_1 < U, \, |V_1 \cap W|=1,$ and
therefore $|V_2|=p^3$.
Let $U'$ be a $\B$-subgroup of order $p^4$ such that
$V_1 < U'$. Then $U \ne U',$ and thus there exists
$u \in U' \setminus U$. Let $T \in \Bs(\B)$ such that $u \in T$.
Then $W+u \subseteq T \subseteq U',$ implying that
$W \le  U'$. We conclude that $U \cap U'=V_1+W=V_2$.

Let  $x^{U'}$ denote the restriction of $x$ to the $\B$-subgroup
$U'$. Clearly, $C_{U'}(x^{U'})=V_2,$ and thus by Lemma~\ref{L-center1}, $\B_{U'}$ is a $V'/W'$-wreath product for
some $\B$-subgroups $0 < W' < V'< U',$ where $|W'|=p$ and
$|V'|=p^3$.  Let $T \in \Bs(\B)$ such that
$T \subseteq U' \setminus V'$ and $T \not\subset U$.
Then using that $T$ contains a $W$-coset and $W' \le \rad(T),$ it 
follows that $T$ contains a $(W+W')$-coset (recall that $W' \ne W$ as 
$W$ is not an $\A$-subgroup). This together with Proposition~\ref{L-p-S2}(ii)
yields that $W \le \rad(T) \le V',$ and thus $W <  V'$. On the other hand it is clear that $V_1 < V',$ and this with the previous
observation yields $V'=W+V_1=V_2$. In particular, $V_2$ is
a $\B$-subgroup. Since $U$ is not a $\B$-subgroup, the S-ring
$\B_{H/V_2} \cong \Q \Z_p \wr \Q \Z_p,$ and this implies that
$U'$ is the only $\B$-subgroup which has order $p^4$ and
contains $V_2$. This property will be used later.  Note also that,
\begin{equation}\label{Eq-W'+W}
W'+W \le \rad(T) \text{ for all } T \in \Bs(\B) \text{ such that }
T \subset U' \setminus V_2.
\end{equation}

Next, let us consider the S-ring $\B_{H/V_1}$.
Let $T \in \Bs(\B)$ such that $T \cap U \ne \emptyset$ and
$T \not\subset U$. The latter condition implies that $T$ contains
some $W$-coset, and thus $|T/V_1| \ge p^2$. It follows by
Proposition~\ref{HM5}, that $T/V_1$ is equal to a $U'/V_1$ coset.
It is easy to see that $\A_{U'/V_1} \cong \Q \Z_p \wr \Q \Z_p;$
and we conclude that
 $\B_{H/V_1} \cong \Q Z_p \wr \Q \Z_p \wr \Q \Z_p$.

Now, fix a coset $U'+v_1$ distinct from $U',$ and let
$T_1 \in \Bs(\B)$ such that $T_1 \subseteq U'+v_1$.
Let $T \in \Bs(\B)$ be another basic set contained in $U'+v_1$.
Since $\B_{H/V_1} \cong \Q Z_p \wr \Q \Z_p \wr \Q \Z_p,$
it follows that both $T_1 \cap V_1+v_1$ and $T \cap V_1+v_1$ are nonempty. Choose $u_1 \in T_1 \cap V_1+v_1$ and $u_2 \in T \cap V_1+v_1$. Then,  $u_2-u_1 \in V_1 \le \O(\B),$ hence by Eq.~\eqref{Eq-eT}, $T_1+u_2-u_1=T$. Thus,
$\rad(T)=\rad(T_1),$ and combining this with Theorem~\ref{T-1st-multi}, we conclude that $\rad(T')=\rad(T_1)$ for any basic set
$T' \not\subset U'$. Since $\B$ is indecomposable, it follows
that $\rad(T_1)$ is trivial.
This together with Lemma~\ref{L-p-S2}(ii) yields that
\begin{equation}\label{Eq-coset2}
|T_1 \cap V_1+v|=1 \text{ for all }  v \in U'+v_1.
\end{equation}
Indeed, if $u_1,u_2 \in T_1 \cap V_1+v,$ then for $W=\sg{u_1-u_2},$ 
$|T_1 \cap W+u_2| \ge 2,$ and so $W \le \rad(T_1),$ a contradiction. 
Thus, $|T_1 \cap V_1+v| \le 1,$ and since $|T_1 \cap V_1+v| \ge 1$ also holds, see the above paragraph, Eq.~\eqref{Eq-coset2} follows.

By Eq.~\eqref{Eq-W'+W}, we can choose a subgroup $W_1 < V_1$ which satisfies the following property: for all $T \in \Bs(\B), \, T \subset U' \setminus V_2,$ either
$|T|=p^3$ or  $W_1 \not< \rad(T)$. Notice that, this property
implies that $\B_{U'/W_1} \cong \Q Z_p \wr \Q \Z_p \wr \Q \Z_p$.

Let us consider $x^{H/W_1},$ the automorphism of $H/W_1$
induced by $x$ acting on $H/W_1$. By Lemma~\ref{L-kernel},
$x$ cannot be in the kernel of $\Aut(\B)$ acting on $H/W_1,$
and so $x^{H/W_1} \ne \id_{H/W_1}$. Then,
$C_{H/W_1}(x^{H/W_1})=U/W_1,$ and thus by Lemma~\ref{L-center1}, $\B_{H/W_1}$ is an $(X/W_1)/(Y/W_1)$-wreath product for some $\B$-subgroups $X$ and $Y,$ where
$|X|=p^4,$ $|Y|=p^2$ and $W_1 < Y < X$.
Notice that, $X \cap U'$ is a $\B$-subgroup of order $p^3$.
On the other hand, since $\B_{U'/W_1} \cong \Q Z_p \wr \Q \Z_p \wr \Q \Z_p,$ $V_2$ is the only $\B$-subgroup in $U'$ that has order $p^3$ and contains $W_1$.
Using this and that $W_1 < X \cap U',$ it follows that
$V_2=X \cap U'$. We have shown above that
$U'$ is the only $\B$-subgroup containing $V_2,$ hence $X=U'$.
Thus, $W_1 < Y < U',$ and using again that
$\B_{U'/W_1} \cong \Q Z_p \wr \Q \Z_p \wr \Q \Z_p,$ we can see
that $V_1$ is the only $\B$-subgroup in $U'$ that
has order $p^2$ and contains $W_1,$ and so $Y=V_1$.
To sum up, $\B_{H/W_1}$ is an
$(U'/W_1)/(V_1/W_1)$-wreath product.
Recall that, $T_1$ is a basic set contained in the coset
$U'+v_1$. Then, $T_1$ satisfies
$V_1/W_1 < \rad(T_1/W_1)$. In particular, for any $u \in T_1,$
$|T_1 \cap (V_1+u)| \ge |T_1/W_1 \cap (V_1+u)/W_1| = p,$ which,
however, contradicts Eq.~\eqref{Eq-coset2}.
This completes the proof of the lemma.
\end{proof}

\begin{lem}\label{L-p7}
Let $\B$ be an indecomposable non-CI $p$-S-ring
over a group $H \cong \Z_p^5$ such that $|\O(\B)| \ge p^2$.
Then $|N_{\Aut(\B)}(H_R)| \ge p^7$.
\end{lem}

\begin{proof}
Let $N=N_{\Aut(\B)}(H_R)$. Assume to
the contrary that $|N| < p^6$. Let
$K \le N$ be the subgroup given
in Proposition~\ref{P-regular-R}.
The stabilizer $(KH_R)_0 \le N_0,$
hence we can write
$$
\frac{|K| |H_R|}{|K \cap H_R|}=|KH_R| =
|(KH_R)_0| \cdot |H| \le p \cdot |H|.
$$
Since $K \ne H_R,$ we can choose a non-identity element
$x \in (KH_R)_0$. Then the above inequality yields
$|C_H(x)| \ge |K \cap H_R| \ge |K|/p=p^4$.
This, however, contradicts Lemma~\ref{L-center2}.
\end{proof}

The key step in proving Theorem~\ref{32} will be to show that, if
$\A$ is non-CI, then $\Aut(\A) \cap \Aut(H)$ contains a subgroup
$L$ such that $|L|=p^2$ and $|C_H(L)|=p^3$.
The next three lemmas are devoted to the arising S-ring $V(L,H)$.

\begin{lem}\label{L-L1}
Assuming Hypothesis~\ref{hyp},
suppose that $\A$ is indecomposable,
and let $L \le \Aut(\A) \cap \Aut(H)$ such that $|L|=p^2$
and $|C_H(L)|=p^3$. Then the following conditions hold:
\begin{enumerate}[(i)]
\item The S-ring $V(H,L)$ is indecomposable.
\item Let $T$ be a basic set of $V(H,L)$ such that $|T|>1$.
Then $T$ is equal to an $X$-coset for some subgroup
$X < C_H(L)$ of order $|X|=p^2$.
\item Let $T$ and $T'$ be two basic sets of $V(H,L)$ of size $p^2$ for  which $\sg{T,C_H(L)} \ne \sg{T',C_H(L)}$. Then
$\rad(T) \ne \rad(T')$.
\end{enumerate}
\end{lem}

\begin{proof}
Let $N=N_{\Aut(\A)}(H_R)$ and $N_0$ be the stabilizer of
$0$ in $N$. Note that, $L \le N_0$ and $\A=V(H,N_0)$.
Furthermore, we let $\B=V(H,L)$ and $U=C_H(L)$.
\medskip

(i): Assume to the contrary that $\B$ is a $X/Y$-wreath product, where $X > Y, \, |X|=p^4$ and $|Y|=p$.
Let $x \in \Aut(H)$ be defined
by $v^x=v$ for all $v \in X,$ and $v^x=v+v_1$ for all
$v \in H \setminus X,$ where $v_1 \in Y$ is a fixed nonzero element.  Then $C_H(x)=X$ and $x \in \Aut(\B)$. Also, as $L \le N_0,$
$\B=V(H,L) \supseteq V(H,N_0)=\A,$ and thus
$\Aut(\B) \le \Aut(\A)$ . In particular, $x \in \Aut(\A),$ which
contradicts Lemma~\ref{L-center2}.
\medskip

(ii): As $|L|=p^2,$ $|T| \le p^2,$ and if $|T| < p^2,$ then $L_v$ is nontrivial, where $v\in T,$ and $L_v$ is the stabilizer of $v$ in $L$.
Then for $x\in L_v,$ $v\in C_V(x),$ and thus
$|C_V(x)| \ge |\sg{U,v}|=p^4,$ which contradicts Lemma~\ref{L-center2}, recall that $\B$ is indecomposable. We deduce that
$|T|=p^2$. Note that, since there is no basic set of size $p,$
it follows that every $\B$-subgroup of order $p^2$ must be contained
in $U$. This fact will be used later.

Let $V < H$ be a $\B$-subgroup such that $U < V$ and $|V|=p^4$.
If $T \subset V,$ then it easy to see that $T$ is equal to a coset of a subgroup of $U$ of order $p^2,$ and (ii) follows.

Now, suppose that
$T \not\subset U$.
By Lemma~\ref{L-p-S2}(iii), $U \cap \rad(T) \ne \{0\},$ and thus we can choose $W < U$ such that $|W|=p$ and $W \le \rad(T)$.
Let us consider the S-ring $\B_{H/W}$.
Then $T/W$ is a basic set of $\B_{H/W}$ of size $|T/W|=p$.
Denote by $L^{H/W}$ the subgroup of $\Aut(\B_{H/W})$ induced
by $L$ acting on $H/W$.
By Lemma~\ref{L-kernel}, $|L^{H/W}|=p^2$. It follows from
this and Proposition~\ref{HM2} that
$T/W$ cannot generate $H/W$. This together with the fact that
$W \le \rad(T)$ shows that $\sg{T} \ne  H,$ and thus $|\sg{T}|=p^3$
or $|\sg{T}|=p^4$. If $|\sg{T}|=p^3,$ then it is easily seen that $T$ is equal to a coset of a subgroup of $U,$ and so (ii) follows.

Assume that $\sg{T}=p^4,$ and let $V'=\sg{T}$.
We show below that this case cannot occur.
If $U< V',$ then it is easy to see that $T$ is equal to a coset of a
subgroup of $U,$ contradicting that $\sg{T}=p^4$.
Thus, $|U \cap V'|=p^2,$ and $H$ can be expressed as the internal
direct sum $H=V' + X$ for some subgroup $X < U, \, |X|=p$.
Note that, as $X \le \O(\B),$ it follows from
Lemma~\ref{L-tensor} that $\B=\B_{V'} \otimes \B_X$.

Let $Y < V'$ be a $\B$-subgroup of
order $p^3$ such that $U \cap V' < Y$. Since $\sg{T}=V',$
$T \not\subset Y$. The radical $\rad(T) \ne U \cap V',$ for otherwise,
$T$ cannot generate $V'$. It follows that the basic sets of $\B$ contained in $V' \setminus Y$ are in the form
$k(T+u)$ for some $k \in \{1,\ldots,p-1\}$ and some
$u \in U\cap V'$. Since $W \le \rad(T),$ we obtain that
$\B_{V'}$ is a $Y/W$-wreath product. This implies that
$\B=\B_{V'} \otimes \B_X$ is a $(Y+X)/W$-wreath product, which contradicts (i).
\medskip

(iii): Assume to the contrary that $\rad(T)=\rad(T')$ and
$\sg{U,T} \ne \sg{U,T'}$. Let $X=\rad(T),$ and let us consider the
S-ring $\B_{H/X}$.
By (ii), both $T$ and $T'$ are $X$-cosets. Since $\sg{U,T} \ne \sg{U,T'},$ we find that the elements $T/X$ and $T'/X$ generate a subgroup of $H/X$ of order
$p^2,$ and this subgroup intersects $U/X$ trivially.
We conclude that $\B_{H/X} \cong \Q\, \Z_p^3$. Consequently, every basic set of $\B$ is contained in some $X$-coset.
This together with (ii) shows that $\B$ is an $U/X$-wreath product,
which contradicts (i).
\end{proof}

\begin{lem}\label{L-L2}
With the notation of Lemma~\ref{L-L1}, $|\Aut(V(H,L))|=p^8$ .
\end{lem}

\begin{proof}
As in the previous lemma, we let $\B=V(H,L)$ and $U=C_H(L)$.
We start with fixing a suitable $\B$-basis.
Fix an element $v_1 \in H \setminus U,$ and another
$v_2 \in H \setminus \sg{U,v_1}$. For $i=1,2,$ Let $T_i \in \Bs(\B)$ such that $v_i \in T_i$. By Lemma~\ref{L-L1}(ii)-(iii),
$T_i-v_i$ is a subgroup of $U$ of order $p^2$.
It will be convenient to denote these subgroups by $U_\infty$ and
$U_0,$ namely, we let $T_1=U_\infty+v_1$ and $T_2=U_0+v_2$.
Then $|U_\infty \cap U_0|=p$. Let $v_4 \in U_\infty \cap U_0, \,
v_4 \ne 0$. Then, there exist $x, y \in L$ satisfying $v_2^x-v_2=v_1^y-v_1=v_4$. Let $v_3=v_1^x-v_1$ and $v_5=v_2^y-v_2$. We prove next that $\sg{v_1,v_2,v_3,v_4,v_5}=H$ and $\sg{x,y}=L$.
For the first part it is enough to show that $\sg{v_3,v_4,v_5}=U$.
Now, $v_1^x \in T_1=U_\infty+v_1,$ hence $v_3 \in U_\infty$.
Suppose for the moment that $v_3=k v_4$ for some integer
$k$. Then
$(k v_2)^x-kv_2=kv_4=v_3=v_1^x-v_1,$
implying that $k v_2 - v_1$ is fixed by $x,$ and hence
$k v_2-v_1 \in C_H(x)=U,$ which is impossible. We conclude that
$U_\infty=\sg{v_3,v_4}$.
We obtain by a similar argument that $U_0=\sg{v_4,v_5},$ and therefore, $\sg{v_3,v_4,v_5}=U_\infty+U_0=U$.
For the second part, if $y=x^m$ for some integer $m,$ then
we can write  $v_2+v_5=v_2^{y}=v_2^{x^m}=v_2+m v_4,$
contradicting that $\sg{v_4,v_5}=U_0$ has order $p^2$.
Thus, $\sg{x,y}=L,$ as required. It is clear that
$(v_1,\ldots,v_5)$ is a $\B$-basis.

Let $V=\sg{v_1,v_2}$.
Then $H$ can be written as the internal direct sum $H=V+U$.
For $w\in H,$ let $w_V$ and $w_U$ denote the projection of
$w$ into $V$ and $U,$ resp.
Furthermore, let $I=GF(p) \cup \{\infty\};$ and for
$i \in I,$ define the elements $\hat{v}_i \in V,$ and subgroups
$U_i\le U$ as follows
$$
\hat{v}_i=\begin{cases} \
v_1 & \text{if } i=\infty \\
iv_1+v_2 & \text{otherwise},
\end{cases} \quad
U_i=\begin{cases} \
\sg{v_3,v_4} & \text{if } i=\infty \\
\sg{iv_3+v_4,iv_4+v_5} & \text{otherwise}.
\end{cases}
$$

Let $G=\Aut(\B)$ and $G_w$ be the stabilizer of an element $w$ in
$G$. Observe that, the lemma is equivalent to show that
$|G_0|=p^3$. We are going to derive this in six steps.
\medskip

\noindent{\bf Claim (a).} {\it The basic sets of $\B$ not contained in
$U$ are in the form
\begin{equation}\label{Eq-basic1}
U_i+j \hat{v}_i+u, \; i\in I, \; j \in GF(p) \setminus \{0\}, \; u \in U.
\end{equation}
}

By definition, the basic sets in question are equal to the $L$-orbits
$w^L, \, w \in H \setminus U,$ where $L=\sg{x,y}$. Now,
$w=j \hat{v}_i+u$ for some $j \in GF(p) \setminus \{0\}, \, i \in I$ and
$u\in U$. A direct computation yields that the $L$-orbit
$\hat{v}_i^L=U_i+\hat{v}_i$. This together with the fact that
$u \in C_H(L)$ yields Claim~(a).
\medskip

Let $\Fun_0(V,U)$ denote the set of all functions $F : V \to
U$ such that $F(0)=0$.  For $F \in \Fun_0(V,U),$ we define the permutation $g_F \in \Sym(H)$ as follows:
\begin{equation}\label{Eq-gF}
w^{g_F}=w+F(w_V), \; w \in H,
\end{equation}
where $w_V$ denotes the projection of $w$ to $V$ (recall that,
we have $H=V+U$).
\smallskip

\noindent{\bf Claim (b).} {\it For every $g \in G_{0},$
$g=g_F$ for some $F \in \Fun_0(V,U)$.}
\smallskip

Since $U$ is a $\B$-subgroup, the set $H/U$ form a
block system for $G$. Let us consider $g^{H/U},$ the permutation of $H/U$ induced by $g$ acting on $H/U$.
Then, $g^{H/U} \in \Aut(\B/U)$. It is easy to see that
$\B/U=\Q \, H/U,$ and thus we get that $g^{H/U}=\id_{H/U}$.
Equivalently, $g$ fixes setwise every $U$-coset.
On the other hand, by Eq.~\eqref{Eq-thin}, $g$ centralizes $u_R$
for all $u \in U$. These two facts imply Claim~(b).
\smallskip

\noindent{\bf Claim (c).} {\it For every $F \in \Fun_0(V,U),$
$g_F \in G_0$ if and only if the following conditions hold:
\begin{equation}\label{Eq-cond3}
F(v+\hat{v}_i)-F(v) \in U_i \text{ for all } v \in V \text{ and } i \in I.
\end{equation}
}

Let $F \in \Fun_0(V,U)$. By definition, $g_F \in G_0$ if and only if
$g_F \in \Aut(\Cay(H,T))$ for any basic set $T\in  \Bs(\B)$.
It is clear that $g_F$ centralizes $u_R$ for all $u\in U$.
Hence, $g_F \in \Aut(\Cay(H,T))$ whenever $T \subset U$.
Now, suppose that $T \not\subset U$. By Claim~(a),
$T=U_i+j \hat{v}_i+u$ for some $j\in GF(p) \setminus \{0\}, \,
i\in I$ and $u \in U$.
Therefore,  $g_F \in \Aut(\Cay(H,T))$ if and only if
$$
(U_i+j\hat{v}_i+u+w)^{g_F}=U_i+j\hat{v}_i+u+w^{g_F} \text{ for all
} w \in H.
$$
By Eq.~\eqref{Eq-gF}, this reduces to
$$
U_i+j\hat{v}_i+w+u+F(j\hat{v}_i+w_V)=U_i+j\hat{v}_i+w+u+F(w_V).
$$
Equivalently, $F(v+j \hat{v}_i)-F(v) \in U_i \text{ for all } v \in V$.
Since,
$$
F(v+j \hat{v}_i)-F(v)=
\sum_{k=1}^{j}\big( \, F(v+k\hat{v}_i)-F(v+(k-1)\hat{v}_i)\, \big),
$$
it follows that 
$F(v+j \hat{v}_i)-F(v) \in U_i \text{ for all } v \in V$ if and only if 
$F(v+\hat{v}_i)-F(v) \in U_i \text{ for all } v \in V,$
and Eq.~\eqref{Eq-cond3} follows.
\smallskip

\noindent{\bf Claim (d).} {\it If $i,j,k \in I$ are pairwise distinct,
and $u_1,u_2,u_3 \in U$ are arbitrary elements, then
$|U_i+u_1\cap U_j+u_2 \cap U_k+u_3|=1.$
}
\smallskip

It is not hard to show that Claim~(d) follows from
$U_i \cap U_j \cap U_k=\{0\}$.
Let $\alpha v_3+\beta v_4+\gamma v_5 \in U_i \cap U_j \cap U_k$.
Suppose at first that none of $i,j$ and $k$ is equal to $\infty$.
Then, using the definition of the subgroups $U_i, U_j$ and $U_k,$
we find $\alpha_1, \alpha_2$ and $\alpha_3$ in $GF(p)$ such
that
\begin{eqnarray*}
\alpha &=& \alpha_1 i \; = \;  \alpha_2 j \; = \; \alpha_3 k \\
\beta &=& \alpha_1+\gamma i \; = \; \alpha_2+\gamma j \; = \;  \alpha_3+\gamma k.
\end{eqnarray*}
Using also that $i,j$ and $k$ are pairwise distinct, we deduce that
$\beta=\gamma (i+j)=\gamma (i+k) = \gamma (j+k),$ and hence
$\alpha=\beta=\gamma=0,$ and so $U_i \cap U_j \cap U_k=\{0\}$.

Now, suppose that $k=\infty$.  Since $U_\infty=\sg{v_3,v_4},$
$\gamma=0,$ and
$\alpha v_3+\beta v_4=\alpha_1(i v_3+v_4)=\alpha_2(j v_3+v_4)$.
Since $i \ne j,$ it follows that
$\alpha_1=\alpha_2=\alpha=\beta=0,$ and $U_i \cap U_j \cap U_k=\{0\},$ as required.
\smallskip

\noindent{\bf Claim (e).} {\it $|G_0 \cap G_{v_1}| \le p$.}
\smallskip

Let $g \in G_0 \cap G_{v_1}$. By Claim~(b), $g=g_F$ for some
$F \in \Fun_0(V,U)$. Notice that, $F(0)=F(v_1)=0$.
Let us consider the image $F(2v_1)$.
Then, we can express $2v_1$ as $2v_1=v_1+\hat{v}_\infty,$ hence
by Eq.~\eqref{Eq-cond3}, $F(2v_1)-F(v_1)\in U_\infty$. Also,
$v_1+v_2=2v_1+\hat{v}_{-1}$ and $v_2=2v_1+\hat{v}_{-2},$ and using again Eq.~\eqref{Eq-cond3},
we find $F(v_1+v_2)-F(2v_1) \in U_{-1}$ and $F(v_2)-F(2v_1) \in U_{-2}$.
All these yield
\begin{equation}\label{Eq-cond4}
F(2v_1) \in  U_\infty+F(v_1) \; \cap \; U_{-1}+F(v_1+v_2) \; \cap \;
U_{-2}+F(v_2).
\end{equation}

On the other hand, $F(v_2) \in U_0 \cap U_{-1}=\sg{-v_4+v_5}$
and $F(v_1+v_2) \in U_0 \cap U_1=\sg{v_4+v_5}$.
Using also that $F(v_1+v_2)-F(v_2) \in U_\infty=\sg{v_3,v_4},$
we find $F(v_2)=\alpha (-v_4+v_5)$ and $F(v_1+v_2)=\alpha (v_4+v_5)$ for some $\alpha \in GF(p)$.
Substitute these in Eq.~\eqref{Eq-cond4}. After a
direct computation we find $F(2v_1)=2 \alpha v_3$.
This shows that the orbit of $2v_1$ under the group $G_0 \cap G_{v_1}$ has size at most $p$. Therefore,
$|G_0 \cap G_{v_1}| \le |G_0 \cap G_{v_1} \cap G_{2v_1}| \cdot p,$
and to derive Claim~(e) it is enough to show that
$|G_0 \cap G_{v_1} \cap G_{2v_1}|=1$.

Now, choose $g \in G_0 \cap G_{v_1} \cap G_{2v_1}$.
Then, $F(0)=F(v_1)=F(2v_1)=0,$ thus applying Eq.~\eqref{Eq-cond3} to $F(iv_1+v_2), \, i \in GF(p),$ and using Claim~(d),
we find $F(i v_1+v_2) \in U_i \cap U_{i-1} \cup U_{i-2}=\{0\}$.
Therefore, $F(iv_1+v_2)=0$ for all $i \in GF(p)$.
In particular, $F(v_2)=F(v_1+v_2)=F(2v_1+v_2)=0,$ and we
can repeat the same argument to have $F(iv_1+2v_2)=0$ for all
$i \in GF(p)$. Since $V=\sg{v_1,v_2},$
 the process can be continued to cover all $v \in V,$ and this leads to that $F(v)=0$ for all $v \in V,$ that is, $g=\id_H,$ as required.
\smallskip

\noindent{\bf Claim (f).} {\it $|G_0|=p^3$.}
\smallskip

The $G_0$-orbit of $v_1$ is the basic set $U_\infty+v_1$.
This together with Claim~(e) shows that $|G_0| \le p^2 \cdot
|G_0 \cap G_{v_1}| \le p^3$. To settle Claim~(f) it is enough to
find a non-trivial automorphism $g \in G_0 \cap G_{v_1}$.
We claim that $g_F$ is such an automorphism where $F$ is defined
as follows:
$$
F(iv_1+jv_2)=i(i-1)v_3+(2i-1)j v_4+j^2 v_5, \ \  i,j \in GF(p).
$$
Then, $F(0)=0,$ and by Claim~(c), we have $g_F \in G_0$ if
$F$ satisfies the conditions in Eq.~\eqref{Eq-cond3}.
This can be verified directly. After letting
$v=iv_1+jv_2$ and using the above definition of
$F,$ we compute for $k \in GF(p),$
\begin{eqnarray*}
F(v+\hat{v}_\infty)-F(v) &=& 2i v_3+2j v_4, \\
F(v+\hat{v_k})-F(v) &=& (2i-1+k)(k v_3+v_4)+(2j+1)(k v_4+v_5).
\end{eqnarray*}
These show that the conditions in Eq.~\eqref{Eq-cond3} hold, and
$g_F \in G_0$. Also, $F(v_1)=0$  and $F(2v_1)=2v_3,$ and hence
$g_F$ is a non-trivial element in $G_0 \cap G_{v_1}$. This completes
the proof of the lemma.
\end{proof}

\begin{lem}\label{L-L3}
With the notation of Lemma~\ref{L-L1}, $V(H,L)$ is a CI-S-ring.
\end{lem}

\begin{proof}
We keep all notations from the previous proof, that is,
$$
L=\sg{x,y}, \; G=\Aut(V(H,L)), \;  U=\O(V(H,L))=\sg{v_3,v_4,v_5}, \; V=\sg{v_1,v_2}.
$$
In addition, let $N=N_G(H_R)$.
In view of Lemma~\ref{HM4}, it is enough to show that all
regular subgroups of $G$ isomorphic to $H$ are conjugate in $G$.
First, the number of subgroups of $G$ that are conjugate to $H_R$ is equal to
the index $|G:N|$. By Lemma~\ref{L-L2}, $|G|=p^8,$ and since
$H_RL \le N,$ it follows that $|N| \ge p^7$. If $G=N,$ then for every non-trivial element
$z \in G_0 \cap G_{v_1},$ $C_H(z)=\sg{v_1,v_3,v_4,v_5},$ contradicting Lemma~\ref{L-center2}.
Thus, $|N|=p^7,$ and there are exactly
$p$ subgroups of $G$ that are conjugate to $H_R$.
Therefore, to finish the proof it is sufficient to show that there are exactly $p$ regular subgroups of $G$ isomorphic to $H$.
Note that, we have $L=N_0$.

Let $K \le G$ be any regular subgroup isomorphic to $H_R$
such that $K \ne H_R$. Let $M=\sg{K,H_R}$. Since $K \ne H_R,$ $|M| \ge p^6$. This implies that $|M \cap L|>1$.
Indeed, if $|M|=p^6,$ then $H_R \trianglelefteq M,$ and hence
$M_0 \ne 1$ and $M_0 \le N_0=L$. If $|M| > p^6,$ then
$|M \cap L|>1$ follows because $|L|=p^2$ and $|LM| \le |G|=p^8$. Using also that
$K \cap H_R \le Z(M)$ and Lemma~\ref{L-center2}, we deduce that
$|K \cap H_R|\le p^3$.   On the other hand, since $K$ is regular and abelian, it follows that
$Z(G) \le K,$ and hence $K \cap H_R=U_R$. Note that, we have proved  that every regular subgroup isomorphic to $H$ intersects $H_R$ at $U_R,$ unless it is equal to $H_R$.
This fact will be used in the next paragraph.

We claim that $K \le N$. Suppose to the contrary that
there exists some $g \in K \setminus N$.
Then $g=z_1 v_R$ for some $z_1\in G_0 \setminus L$ and
$v \in H \setminus U$. On the other hand,
$|N \cap K|\ge p^4,$ and
and thus $K$ contains an element in the form
$z_2 w_R, \, z_2\in L$ and $w\in H \setminus U$. Clearly, since $U_R \le K,$ the element
$w$ cannot be in $U$. Then,
$z_1z_2(v_R)^{z_2}w_R=z_1v_R z_2w_R=z_2w_R z_1v_R=
z_2z_1(w_R)^{z_1}v_R$. By Lemma~\ref{L-L2}, $G_0$ is abelian, and
we get $(w_R)^{z_1}=(v_R)^{z_2}w_R(v_R)^{-1}=(v^{z_2}+w-v)_R$.
Thus, $(w_R)^{z_1} \in H_R,$ and since $w\notin U,$ $(w_R)^{z_1} \notin U_R,$ and
$|H_R^{z_1} \cap H_R|\ge p^4$. Now, it follows by the previous
paragraph that $H_R^{z_1}=H_R,$ and hence $z_1 \in L,$ a contradiction.

Now, there exist $z_1,z_2 \in L$ such that
$$
K=\big\langle   z_1(v_1)_R, \, z_2(v_2)_R, U_R \big\rangle.
$$

Recall that, the $L$-orbit of $v_1$ is in the form
$v_1^L=U_\infty+v_1,$ and the $L$-orbit of $v_2$ is in the form
$v_2^L=U_0+v_2$. Let $u=v_1^{z_2}-v_1$. Clearly, $u \in U_\infty$.
Then, since $z_1(v_1)_R$ and
$z_2(v_2)_R$ commute, $0^{z_1(v_1)_R z_2(v_2)_R}=v_1^{z_2}+v_2$
and $0^{z_2(v_2)_R z_1(v_1)_R}=v_2^{z_1}+v_1,$ it follows that $u=v_2^{z_1}-v_2$.
This shows that $u\in U_0$ also holds, and hence $u \in U_\infty \cap U_0 =\sg{v_4}$.
Now, as $L$ is regular on both orbits $v_1^L$ and $v_2^L,$ the automorphisms $z_1$ and $z_2$
are uniquely determined by $u,$ and thus $K$ is determined as well.
This yields that there are exactly $p$ regular subgroups of $G$
isomorphic to $H$. This completes the proof of the lemma.
\end{proof}

\begin{lem}\label{L-UW}
Assuming Hypothesis~\ref{hyp}, suppose that $\A$ is indecomposable,
and let $U$ and $W$ be
$\A$-subgroups such that $W < \O(\A) < U, \, |W|=p$
and $|U|=p^4$. Then $\A$ has a basic set $T$ such that
\begin{equation}\label{Eq-T}
T \subset H \setminus U, \; 1 < |T| \le p^2 \text{ and }
W \not\le  \rad(T).
\end{equation}
\end{lem}

\begin{proof}
Since $\A$ is indecomposable, there exists a basic set
$T_1 \subset H \setminus U$ such that $W \not\le \rad(T_1)$.
It is clear that $|T_1|>1$. 
We have to show that $|T_1| \le p^2$.  To the contrary
assume that $|T_1| \ge p^3$. If $|T_1|=p^4,$ then $\A$ is
decomposable, see Proposition~\ref{HM5}, thus $|T_1|=p^3$.
This together with $|\O(\A) \cap U|=|\O(\A)| \ge p^2$ gives
$ |\O(\A) \cap U| \cdot |T_1| > p^4=|H|/p,$ and
we can apply Lemma~\ref{L-p-S2}(iii) to obtain that
$\O(\A) \cap \rad(T_1) \ne \{0\}$. Let $W' \le \O(\A) \cap \rad(T_1) $ such that $|W'|=p$. Since $W \not\le \rad(T_1),$ we get, using
Eq.~\eqref{Eq-eT}, pairwise distinct
basic sets in the form $T_1+w, \, w \in W$. As the union of the
latter basic sets is equal to the coset $U+v_1,$ it follows that $W' \le \rad(T)$ for all  $T\in \Bs(\A)$ with $T \subset U+v_1$. This together with Theorem~\ref{T-1st-multi} yields  that $W' \le \rad(T)$ for all  $T\in \Bs(\A)$ with $T \not\subset U,$ that is, $\A$ is a $U/W'$-wreath product, a  contradiction.
\end{proof}

Everything is prepared to settle the main result of the section.

\begin{proof}[Proof of Theorem~\ref{32}]
Let $U=\O(\A),$ $N=N_{\Aut(\A)}(H_R)$ and $N_0$ be the stabilizer
of $0$ in $N$. Assume to the contrary that $\A$ is a non-CI-S-ring.
We prove first the following:
\begin{equation}\label{L}
\text{There exists } L \le N_0 \text{ such that }
|L|=p^2 \text{ and } |C_H(L)|=p^3.
\end{equation}

If $|U|\ge p^3,$ then we are done by choosing $L$
to be any subgroup of $N_0$ of order $p^2$ (see also
Lemma~\ref{L-p7}). Thus we assume for the moment that
$|U|=p^2$. Let $K \le N$ be the subgroup given in Proposition~\ref{P-regular-R}, and let $M=(KH_R)_0$. If $|M| \le p^2,$ then
$$
C_H(M) \ge |K \cap H_R|=\frac{|K| \cdot |H_R|}{|K H_R|}=
\frac{|K| \cdot |H_R|}{|M| |H|} \ge p^3.
$$
By Lemma~\ref{L-center2}, $M$ must have order $p^2,$ and therefore, we are done by choosing
$L$ to be $M$.

Let $|M|\ge p^3$.  It follows from Lemma~\ref{L-UW} that there exists a basic set $T$ such that $|T| \le p^2$ and $\rad(T) \ne U$.
Let $v \in T,$ and $M_v$ be the stabilizer of $v$ in $M$.
Then $C_H(M_v) \ge \sg{U,v}$. If $|T|=p$ or the orbit $v^M \ne T,$
then it follows that $|M_v| \ge p^2$. Using this and that
$|\sg{U,v}|=p^3,$ we can choose
$L$ to be any subgroup of $M_v$ of order $p^2$.
Now, suppose that $|T|=p^2,$ and the orbit $v^M=T$.
Choose a non-identity element $x_0 \in M_v,$ and let $u  \in T$ be
an arbitrary element. Then $u=v^x$ for some $x\in M,$ and since $M$ is abelian, we can write $u^{x_0}=v^{xx_0}=v^{x_0x}=
v^x=u$. As a corollary we find $C_H(x_0) \ge \sg{U,T}$.
Clearly, $|\sg{U,T}|\ge p^3;$ in fact, $\sg{U,T}|=p^3$ must hold by Lemma~\ref{L-center2}.  It follows that $T$ is equal to a
$U$-coset, that is, $\rad(T)=U$. This is a contradiction, and
Eq.~\eqref{L} follows.

By Lemma~\ref{L-L3}, the S-ring $V(H,L)$ is a CI-S-ring. Therefore,
$\A \ne V(H,L),$ hence $A > L,$ in particular, $|A|\ge p^3$.
For sake of simplicity we let $V=\O(V(H,L))$. Clearly, $U \le V$ and
$|V|=p^3$. Fix $W_1 < U, \, |W_1|=p$. By Lemma~\ref{L-UW}, there exists a basic set $T_1 \in \Bs(\A)$ such that $1 < |T_1| \le p^2$ and $W_1 \not< \rad(T_1)$. Since every basic set of $V(H,L)$ outside $V$ is a coset of a subgroup of $V$ of order $p^2,$ we find that either
$T_1$ is contained in $H \setminus V,$
$|T_1|=p^2$ and  $\rad(T_1) < V,$ or $T_1 \le V$.
In the former case $V=\rad(T_1)+W_1,$ whereas in the latter case
$V=\sg{T_1,W_1}$ because $W_1 \not< \rad(T_1)$. We conclude that   $V$ is an $\A$-subgroup. This shows that we may
choose the above $T_1$ such that
$T_1 \subset H \setminus V$.
Fix some $v_1 \in T_1$.
Since $|N_0| \ge p^3,$ there exists a non-identity element
$x \in N_0$ such that $v_1^x=v_1,$ and thus
$C_H(x) \ge \sg{U,v_1},$ and so $|C_H(x)| \ge p \cdot |U|$.
This together with Lemma~\ref{L-center2}
shows that $|U|=p^2,$ in particular, $U < V$.
Now, using also that $W_1 < U$ and $W_1 \not< \rad(T_1),$
we infer in turn that $|U \, \cap \, \rad(T_1)|=p,$
$\A_{\rad(T_1)}=\Q C_p \wr \Q C_p,$ and finally that,
$\A_V=(\Q C_p \wr \Q C_p) \otimes \Q C_p$.
 Let $W_2=U \, \cap \, \rad(T_1)$.
It is easy to see that every $\A$-subgroup of order $p^2$ contained in $V$ contains $W_2$.
(In fact, such an $\A$-subgroup intersects $\rad(T_1)$ at exactly $W_2$.)
 Now, apply Lemma~\ref{L-UW} with $W=W_2$.
We obtain that, there exists a basic set $T_2$ of $\A$ such that
$T_2 \not\subset V,\, |T_2|=p^2$ and $W_2 \not< \rad(T_2)$. As before, $\rad(T_2)$ is an $\A$-subgroup of order $p^2$ contained $V,$ contradicting our earlier observation that such an $\A$-subgroup must contain $W_2$. This completes the proof of the theorem.
\end{proof}

\end{document}